\documentclass[reqno, 12pt]{amsart}
\usepackage{graphicx,overpic,epsf,amssymb,amscd,amsfonts}

\newtheorem{thm}{Theorem}[section]

\newtheorem{prop}[thm]{Proposition}
\newtheorem{defn}[thm]{Definition}
\newtheorem{lem}[thm]{Lemma}

\newtheorem{rmk}[thm]{Remark}

\newcommand{\C}{\mathbb{C}}
\newcommand{\R}{\mathbb{R}}

\newcommand{\M}{\mathcal{M}}
\newcommand{\Morse}{{\rm Morse}}
\newcommand{\deldel}[2]{\frac{\partial {#1}}{\partial {#2}}}
\newcommand{\del}[1]{{\partial_{#1}}}
\newcommand{\st}{{\rm st}}

\newcommand{\ev}{{\rm ev}}
\newcommand{\cont}{{\rm cont}}

\newcommand{\indexof}{{\rm index}}
\newcommand{\Crit}{{\rm Crit}}
\newcommand{\be}{\begin{enumerate}}
\newcommand{\ee}{\end{enumerate}}
\begin{document}
%%%%%%%%%%%%%%%%%%%%%%%%%%%%%%%%%%%%%%%%

\title{Correlators and Descendants of Subcritical Stein Manifolds}

\author{Jian He}
\address{Universit\'{e} Libre de Bruxelles, Bruxelles, Belgium}
\email{jianhe@ulb.ac.be}
\date{June, 2012.}

\begin{abstract}
We determine contact homology algebra of a subcritical Stein-fillable 
contact manifold whose first Chern class vanishes. We also compute the
genus-$0$ one point correlators and gravitational descendants 
of compactly supported closed forms of their subcritical Stein fillings.
This is a step towards determining the full potential function of 
the filling as defined in \cite{EliashbergGiventalHofer}. These invariants 
also give a canonical presentation of the cylindrical contact homology. 
With respect to this presentation, we
determine the degree-$2$ differential in the Bourgeois--Oancea exact 
sequence of \cite{Oancea}. As a further application, we proved that if
a K\"{a}hler manifold $M^{2n}$ admits a subcritical polarization 
and $c_1$ vanishes in the subcritical complement, then $M$ is uniruled.
\end{abstract}

\maketitle

\section{Introduction} \label{section: introduction}

The theory of pseudoholomorphic curves has been a 
very active area of research since it was introduced
by Gromov to the study of symplectic manifolds. 
Eliashberg, Givental and Hofer \cite{EliashbergGiventalHofer} 
generalized this to open symplectic manifolds 
with contact type boundaries at infinity. 
These algebraic invariants arise from the structure of the moduli 
spaces of finite energy punctured holomorphic curves with prescribed
asymptotic behaviours. A lot of effort has been made towards the computation
of these invariants. The cylindrical
contact homology of the boundary of a subcritical Stein domain 
$(M^{2n},\partial M)$, whose first Chern class vanishes, was first computed
by Mei-Lin Yau \cite{Yau}. It is also an easy consequence of the results of
Bourgeois and Oancea \cite{Oancea}, relating cylindrical contact homology and
symplectic homology.

\begin{thm}[\cite{Yau}]\label{yau1}
Let $(M^{2n},\partial M)$, $n\geq 2$, be a subcritical Stein domain of finite type, 
and $\xi$ the maximal complex subbundle on $\partial W$. 
If $c_{1}(\xi) = 0$, then 
%\vspace{-0.2cm}
$$HC_{i}(\partial M, \xi) \cong \bigoplus_{m=0}^{\infty} H_{2(n+m-1)-i}(M) \cong \bigoplus_{m=0}^{\infty} H_{i+2-2m}(M,\partial M).$$
\end{thm}

Given a graded vector space $V$, let $\Lambda(V)$ denote the tensor algebra generated by 
$V$ modulo the graded commuting relations, and $V[n]$ the graded vectors space consisting of elements of $V$ with a grading shift of positive $n$.

The conclusion of Theorem \ref{yau1} can be rewritten as 
$$HC(\partial M)\cong \bigoplus_{m=0}^{\infty} H(M,\partial M)[2m-2].$$

Under the same vanishing first Chern class condition, we first determine the rational contact 
homology algebra of $\partial M$, $HC^{\cont}(\partial M)$. In the language of 
\cite{EliashbergGiventalHofer}, this is the full contact homology algebra specialized 
at the origin. 

\begin{thm}\label{contacthomology}
Let $M^{2n}$, $n\geq3$, be a subcritical Stein domain of finite type with $c_1(M) = 0$. 
Then
$$HC^{\cont}(\partial M)\cong 
\Lambda\left(\bigoplus_{m=0}^{\infty} H(M,\partial M)[2m-2]\right).$$
\end{thm}

Essentially this means that the higher differentials in the contact homology algebra, 
coming from rational curves with multiple negative punctures, 
has no affect on homology even though they may be non-trivial on chain level. In fact
the higher differentials are zero provided the existence of a regular equivariant 
perturbation. However this requires more detail understanding of the polyfold 
perturbation theory, and we will not prove it in this paper. 

We then compute the genus-$0$ one point correlators and gravitational 
descendants of the subcritical Stein domain.

\begin{thm}\label{main}
Let $M^{2n}$, $n\geq3$, be a subcritical Stein manifold of finite type with $c_1(M) = 0$. 
Then there exists a canonical isomorphism 
$$\iota\colon HC(\partial M) \longrightarrow \bigoplus_{m=0}^{\infty} H(M,\partial M)[2m-2]$$ 
such that if $a = [\sum_{i=1}^{k} c_{i}\gamma_{i}] \in HC(\partial M)$ is a cycle in 
cylindrical contact homology, 
and $\theta$ is a compactly supported closed form on $M$, then the correlators 
and descendants satisfy
\begin{equation}\label{desc} 
\sum_{i=1}^{k}c_{i}\int_{\mathcal{M}_{\gamma_i}}\ev^*(\theta)\wedge \psi^{m} = \frac{1}{m!}
<\iota_m(a),[\theta]>.
\end{equation} 
\end{thm}
Here the notations are as follows:
\begin{itemize}
\item $\iota$ will be the same isomorphism as in Theorem \ref{yau1}; 
\item $\mathcal{M}_{\gamma}$ is the moduli space of holomorphic planes 
with one interior marked point, and asymptotic to the Reeb orbit $\gamma$;
\item $\ev\colon \mathcal{M}_{\gamma} \rightarrow M$ is the evaluation map at 
the marked point; 
\item $\iota_m\colon HC(\partial M) \rightarrow H(M,\partial M)$ is 
the projection of $\iota$ onto the $m$-th factor of the direct sum, which is then identified with 
$H(M,\partial M)$, ignoring the shift in grading; 
\item $\psi$ is the first Chern class of the tautological line bundle over the moduli 
space of holomorphic planes with one marked point; 
\item $<,>$ is the natural pairing between homology and cohomology of compact support. 
\end{itemize}
The precise definition and invariance of correlators and gravitation descendants will be discussed 
in the next section. Abusing notation, the linear combination of integrals on the left hand side of 
\eqref{desc} is often written simply as $\int_{\M_a}\ev^*(\theta)\wedge \psi^{m}$.

\begin{rmk}
\emph{
In \cite{Yau}, Theorem \ref{yau1} was proved by choosing a suitable 
contact form $\alpha$, where the differential for $HC(\partial M, \alpha)$ 
is explicitly identified with a Morse differential on $M$. It was unknown if an
isomorphism constructed this way is canonical. In other words, if $\beta$ is 
another suitable contact form with isomorphism $\iota'$, and 
$\Phi\colon HC(\partial M, \alpha) \rightarrow HC(\partial M, \beta)$
the natural isomorphism between cylindrical contact homologies, does 
$\iota$ equal $\iota' \circ \Phi$?
Since correlators and descendants behave naturally under change 
of contact form, Theorem \ref{main} therefore implies that $HC(\partial M)$ is canonically
isomorphic to $\bigoplus_{m=0}^{\infty} H(M,\partial M)[2m-2]$. Given a basis 
$\{\theta_i\}$ of $H^*(M,\partial M)$, the $m$-th coordinate in the direct sum of 
a element $a\in HC(\partial M)$ is determined by the values of 
$\{\int_{\M_a}\ev^*(\theta_i)\wedge \psi^{m}\}$.
}
\end{rmk}

In \cite{Oancea}, Bourgeois and Oancea gave a different proof of Theorem \ref{yau1}
using their long exact sequence relating cylindrical contact homology of
$\partial M$ and symplectic homology of $M$. We 
explicitly compute the degree $2$ map $D$ appearing in their exact triangle 
$$SH^+(M) \longrightarrow HC(\partial M) \stackrel{D}{\longrightarrow} HC(\partial M)
\longrightarrow SH^{+}(M).$$

\begin{thm}\label{coherent}
Let $M^{2n}$, $n\geq3$, be a subcritical Stein manifold of finite type with $c_1(M) = 0$. 
Then the map 
$$\iota\circ D\circ \iota^{-1}\colon \bigoplus_{m=0}^{\infty} H(M,\partial M)[2m-2] \rightarrow \bigoplus_{m=0}^{\infty} H(M,\partial M)[2m-2]$$ is given by shifting each factor down one spot,
with ${\rm Ker}(\iota\circ D\circ \iota^{-1}) = H(M,\partial M)[-2]$.
\end{thm}

\begin{rmk}
\emph{Together with Proposition 9 and Remark 19 of \cite{Oancea}, Theorem \ref{coherent}
implies that the isomorphism in the Bourgeois--Oancea proof of Theorem \ref{yau1} is the same
as the canonical isomorphism $\iota$.
}
\end{rmk}

In \cite{BiranCieliebak}, Biran and Cieliebak studied 
\emph{subcritical polarization} of a
K\"{a}hler manifold and asked if 
manifolds admitting subcritical polarizations are
always uniruled. A subcritical polarization is a K\"{a}hler manifold $(M,\omega,J)$ with an
integral K\"{a}hler form $\omega$, together with a smooth reduced complex 
hypersurface $\Sigma$ representing the Poincar\'{e} dual of $k[\omega]$, 
such that the complement of $\Sigma$ is a subcritical
Stein manifold of finite type.  As a consequence of Theorem \ref{main} and a Morse--Bott computation 
of correlators on the normal bundle of $\Sigma$, we obtain a
partial answer to the question of Biran and Cieliebak:

\begin{thm}\label{unirule}
If $(M^{2n},\omega,J,\Sigma,k)$ is a subcritical polarization such that 
$c_{1}(M\setminus \Sigma) = 0$, then $k=1$ and $M$ is uniruled.
\end{thm}

There are only two subcritical polarizations in complex dimension $\leq 2$: 
$(\mathbb{C}P^{1},{\rm pt})$ and 
$(\mathbb{C}P^2, \mathbb{C}P^1)$. Therefore throughout the rest of the paper, we assume $n\geq 3$.
 
This paper is organized as follows:
in section $2$ we review some basic facts of symplectic field
theory, and define the relevant invariants; in section $3$
subcritical Stein manifolds and their Reeb dynamics are described in more detail,
essentially summarizing the previous work of Yau, and we prove Theorems \ref{contacthomology} and 
\ref{main} modulo the technical Proposition \ref{tech}; in section $4$ we apply Theorem \ref{main}  to
subcritical polarizations; in section $5$ we prove Proposition \ref{tech}; and in the last section
we discuss connections with the Bourgeois--Oancea exact sequence and prove Theorem \ref{coherent}.

\section{Symplectic Field Theory and Descendants}\label{definitions}

In this section we will give an extremely brief overview of 
aspects of symplectic field theory and define the invariants
we wish to compute. See \cite{EliashbergGiventalHofer} for a
more complete discussion. Throughout this section we
assume the polyfold theory of Hofer, Zehnder and Wysocki,
\cite{polyfolds}, \cite{polyfolds2}, \cite{polyfolds3}, which forms the
analytical foundation of SFT. In other words, there 
exists a abstract perturbation scheme under which all moduli spaces 
are branched manifolds with boundaries and corners of the expected dimension. 

Let $(V^{2n-1},\xi)$ be a contact manifold with a contact $1$-form $\alpha$, i.e.,
$(d\alpha)^{n-1}\wedge \alpha$ is a volume form and $\xi = \mbox{Ker}(\alpha)$. The 
\emph{Reeb vector field} is the unique vector field $R$ such that
$$ d\alpha (R,-) = 0,\ \ \ \ \ \alpha(R)=1.$$
The flow of the Reeb vector field preserves the contact structure $\xi$. A (possibly multiply covered) 
Reeb orbit $\gamma$ is \emph{non-degenerate} if the linearized Poincar\'{e} return map of the Reeb flow 
has no eigenvalue equal to $1$. For a generic choice of $\alpha$, there are countably many closed Reeb
orbits, all of which are non-degenerate. Let $\kappa_{\gamma}$ denote the multiplicity of the orbit $\gamma$.

\begin{defn}
\emph{
A Reeb orbit is \emph{good} if it is not an even multiple of another orbit $\gamma$ such that the linearized Poincar\'{e}
return map along $\gamma$ has an odd total number of eigenvalues (counted with multiplicity) in the interval $(-1,0)$.}
\end{defn}

\begin{rmk}
\emph{All orbits appearing in this paper are easily seen to be good.
}
\end{rmk}
If $c_1(\xi)=0$, then the Reeb orbits admit a consistent $\mathbb{Z}$-grading.
Define the index of a Reeb orbit $\gamma$ to be $\overline{\mu}({\gamma}) = \mu({\gamma})+(n-3)$, where $\mu({\gamma})$ is the
Maslov index of $\gamma$, computed using a trivialization over a spanning surface for $\gamma$. Since the
first Chern class vanishes, the Maslov index does not depend on the choice of spanning surface.
From now on we will assume $c_1(\xi) = 0$.

The \emph{symplectization} of a contact manifold $(V^{2n-1},\xi, \alpha)$ is 
the manifold $V\times \mathbb{R}$ with the 
symplectic form $d(e^t\alpha)$, where $t$ is the coordinate of $\mathbb{R}$. 
An almost complex structure $J$ on a symplectization 
$(V\times \mathbb{R}, d(e^t\alpha))$ is \emph{compatible} if 
\begin{itemize}
\item $J^2 = -\mbox{Id}$, 
\item $d\alpha(v, Jv)>0$ for all non-zero $v\in \xi$,
\item $J$ is invariant under translation in the $\mathbb{R}$-direction, 
\item $J\xi = \xi$, and $J\del{t} = R$.
\end{itemize}

A \emph{symplectic filling} $(M,\omega)$ of a contact manifold $(V,\xi,\alpha)$ is an open symplectic manifold 
with one open cylindrical end of the form $E = V\times [0,\infty)$.
On the cylindrical end, $\omega|_{E} = d(e^t\alpha)$. The complement of $E$ is compact.
We will often abuse notation and refer to $V$ as $\partial M$.
An almost complex structure $J$ on a filling $(M,\omega)$ is \emph{compatible} if
\begin{itemize} 
\item $J^2=-\mbox{Id}$, 
\item $\omega(v,Jv)>0$ for all non-zero $v\in TW$,
\item on the cylindrical end $E$, $J$ is invariant under translation in the $\mathbb{R}$-direction, 
\item on $V=V\times\{0\}$, $J\xi = \xi$, and $J\del{t} = R$.
\end{itemize}

A \emph{symplectic cobordism} $(M,\omega)$ is an open symplectic manifold with a 
positive cylindrical end $V\times [0,\infty)$ and a negative end $V'\times (-\infty,0]$. 
It is often denoted by $\overrightarrow{VV'}$.
Compatible complex structures are defined in similar fashion. 
 
A \emph{$J$-holomorphic curve} is a map $u$ from a punctured Riemann surface $(\Sigma,i)$
to an almost complex manifold $(M, J)$ such that $$du\circ i = J\circ du.$$ 
Two holomorphic curves $u\colon (\Sigma,i) \rightarrow (M,J)$ and $v\colon (\Sigma',j) \rightarrow (M,J)$ 
are \emph{equivalent} if there is a diffeomorphism $\phi\colon \Sigma\rightarrow \Sigma'$ such that
$u = v\circ \phi$ and $i = \phi^* j$. 
Denote by $\mathcal{M}_{\gamma; \gamma'}$ the moduli space of equivalence classes of holomorphic 
cylinders in the the symplectization $V\times \mathbb{R}$ such that the positive end of the
cylinder is asymptotic to the Reeb orbit cylinder $\gamma\times \mathbb{R}$ and the negative end to 
$\gamma'\times \mathbb{R}$.

The expected dimension of $\mathcal{M}_{\gamma; \gamma'}$ is 
$${\rm dim \ } \mathcal{M}_{\gamma; \gamma'} = \overline{\mu}({\gamma}) -\overline{\mu}({\gamma'}).$$ 
We will assume that $J$ is \emph{regular} (or a suitable abstract perturbation has been performed), 
so that all moduli spaces are branched manifolds of the expected dimension.
Since $J$ is $\mathbb{R}$-invariant, holomorphic curves come in $\mathbb{R}$-families as well.
If $\mathcal{M}_{\gamma; \gamma'}$ is $1$-dimensional, then we can count the number of $\mathbb{R}$-components 
of $\mathcal{M}_{\gamma; \gamma'}$.

\begin{defn}
\emph{
The \emph{cylindrical contact homology} of $V$, $HC(V)$, is the homology of 
the chain complex of the module generated by the good Reeb orbits. 
The differential is given by
$$\partial \gamma = \kappa_{\gamma}\sum_{\gamma'\colon \overline{\mu}({\gamma'})=\overline{\mu}({\gamma})-1} n_{\gamma, \gamma'}\gamma',$$
where $n_{\gamma,\gamma'} = \#(\mathcal{M}_{\gamma; \gamma'}\slash \mathbb{R})$. 
}
\end{defn}
 
As observed in \cite{EliashbergGiventalHofer} and \cite{Yau}, cylindrical contact homology is well defined in certain situations.

\begin{thm}\label{-101}
If $(V,\xi,\alpha)$ has no Reeb orbits of index $-1$, $0$ and $1$, then $\partial^2 =0$, and
$HC(V,\xi)$ is independent of choice of such $\alpha$'s.
\end{thm}

\begin{rmk}
\emph{
To understand the natural isomorphism between cylindrical contact homologies, let $\beta = e^f\alpha$
be another contact form for $(V,\xi)$. In the symplectization of $(V,\alpha)$, the section $V' = \{(p, e^f(p))\colon p\in V\}$ has
induced contact form $\beta$. The symplectization can be regarded as a symplectic cobordism between $(V',\beta)$ and $(V,\alpha)$.
The count of $0$-dimensional moduli spaces of holomorphic cylinders asymptotic to Reeb orbits in $V'$ and $V$ defines a chain 
map which induces isomorphism on homology.
}
\end{rmk}

A slight generalization of cylindrical contact homology is well defined for all $(V,\xi,\alpha)$.
Instead of only cylinders, rational curves with arbitrarily many negative punctures are allowed. Let 
$\mathcal{M}_{\gamma; \gamma_1,\dots, \gamma_k}$ denote the moduli space of genus-$0$ holomorphic
curves with a positive puncture asymptotic to $\gamma$ and several negative
punctures asymptotic to $\{\gamma_{1},\ldots,\gamma_{k}\}$. The expected dimension of
$\mathcal{M}_{\gamma; \gamma_1,\dots, \gamma_k}$ is 
$${\rm dim \ } \mathcal{M}_{\gamma; \gamma_1,\dots, \gamma_k} = 
\overline{\mu}({\gamma}) - \sum_{i=1}^{k}\overline{\mu}({\gamma_i})$$

\begin{defn}
\emph{
The \emph{contact homology algebra} of $V$, $HC^{\cont}(V)$, is the homology of 
the chain complex of the graded commutative algebra generated by the good Reeb orbits.
The differential of a single orbit is given by 
$$\partial \gamma = \kappa_{\gamma}\sum
n_{\gamma; \gamma_1,\dots, \gamma_k} \gamma_{1}\dots \gamma_{k}$$ 
where $n_{\gamma; \gamma_1,\dots, \gamma_k} = \#(\mathcal{M}_{\gamma; \gamma_1,\dots, \gamma_k}\slash \mathbb{R})$,
and the sum is taken over all sets of orbits $\{\gamma_1,\dots,\gamma_k\}$ 
such that $\overline{\mu}({\gamma}) - \sum_{i=1}^{k}\overline{\mu}({\gamma_i})=1$.  The differential
is extended to the entire algebra by Leibniz rule.
}
\end{defn}

\begin{rmk}
\emph{In the language of \cite{EliashbergGiventalHofer}, the contact homology algebra defined here is 
the specialization of their more general $H_*^{\cont}(V)$ at $t=0$.
}
\end{rmk}

\begin{rmk}\label{carpet}
\emph{To simplify exposition we ignored asymptotic markers in our definitions. Instead we hide it in 
the definition of $n_{\gamma; \gamma_1,\dots, \gamma_k}$. Curves in 
$\mathcal{M}_{\gamma; \gamma_1,\dots, \gamma_k}\slash \mathbb{R}$ need to be counted with consistent weights,
and the asymptotic makers contribute a certain combinatorial factor depending on the multiplicities of the asymptotic
Reeb orbits.
}
\end{rmk}

The natural isomorphism between contact homology algebras for different contact forms is a 
direct generalization of
the cylindrical case. The chain map is given by the count of $0$-dimensional moduli spaces of 
genus zero curves with one positive and several negative ends in a cobordism interpolating between 
the contact forms.

\begin{rmk}
\emph{Given a symplectic filling $M$ of $V$, there is a natural \emph{augmentation} on the 
differential graded algebra for the contact homology algebra of $V$, coming from the
count of rigid holomorphic planes in $M$ asymptotic to Reeb orbits on $V$. Using this augmentation
one can defined the \emph{linearized contact differential} and the \emph{linearized
contact homology} of $V$ with respect to the filling $M$. See \cite{surgeryformula} and \cite{Oancea} for detailed 
definitions and discussion. If the cylindrical contact homology of $V$ is well defined, 
and there is no rigid 
holomorphic planes in $M$, then the cylindrical contact homology of $M$ is the same as the
linearized contact homology of $V$ with with respect to $M$. This will be the case for subcritical
Stein fillings.
}
\end{rmk}

\begin{rmk}\label{homologygrade}
\emph{
Both $HC(V)$ and $HC^{\cont}(V)$ are naturally graded by $H_1(V)$, given by the sum of the 
homology classes of the orbits. 
}
\end{rmk}

Suppose $(M,\partial M)$ is a symplectic filling with $c_1(M)=0$.
Let $\mathcal{M}_{\gamma}$ denotes the moduli space of holomorphic planes in $M$
asymptotic to the Reeb orbit $\gamma$ with one interior marked point $x$. It has expected dimension
$${\rm dim \ } \mathcal{M}_{\gamma} = \overline{\mu}({\gamma}) + 2$$
Let $\theta\in H^*(M,\partial M)$ be a compactly supported closed form on $M$, and $\ev\colon u \rightarrow u(x)$ be the evaluation map at the marked point.

\begin{defn}
\emph{
A \emph{genus-$0$ one point correlator} is an integral of the form
$$\int_{\mathcal{M}_{\gamma}}\ev^*(\theta)$$
}
\end{defn}

We may interpret the correlator as the intersection number of 
$\ev(\mathcal{M}_{\gamma})$ with the Poincar\'{e} dual
of $\theta$. However, in general $\mathcal{M}_{\gamma}$ has 
codimension-$1$ boundary strata, therefore
the value of the correlator depends on $\theta$, $J$ and $\alpha$. One way to have
an invariant is to take a linear combination of moduli spaces so that their codimension-$1$ boundaries cancel.

\begin{prop}\label{invariantcor}
Let $(M, \partial M)$ be a symplectic filling with a contact form $\alpha$ on $\partial M$ such that all
Reeb orbits have index at least $2$. If $a = [\sum_{i=1}^{k} c_{i}\gamma_{i}] \in HC_{m}(\partial M)$ is a cycle
in cylindrical contact homology, and $\theta$ is a compactly supported closed $(m+2)$-form, then the value of the linear combination of correlators $$\sum_{i=1}^{k}c_{i}\int_{\mathcal{M}_{\gamma_i}}\ev^*(\theta)$$
is independent of all choices. 
\end{prop}
\begin{proof}
The proof by dimension count is the same as that of Theorem \ref{-101}. 
By the compactness theorem of \cite{compactness},
a codimension-$1$ stratum of $\mathcal{M}_{\gamma}$
consists of $2$-story curves $(u_1, u_2)$, such that $u_1\in \mathcal{M}_{\gamma; \beta_1,\dots, \beta_l}$ 
is a genus-$0$ holomorphic curve in the symplectization $\partial M\times \mathbb{R}$ with several negative
ends; and $u_2$ consists of $l$ holomorphic planes in the filling $M$, one asymptotic to each $\beta_i$. 
The marked point can be located on $u_1$ or on any one of the holomorphic planes 
of $u_2$. Consider
the image of such a stratum under the evaluation map. If the marked point is on $u_1$, 
then it is mapped to infinity (in other words to $\partial M$). Suppose the marked point lies on the
plane asymptotic to $\beta_1$. If $l>1$, then since all orbits have index at least $2$, 
$$\overline{\mu}({\beta_1})\leq \overline{\mu}({\gamma}) - 1 - \overline{\mu}({\beta_2}) - \dots -\overline{\mu}({\beta_l})
\leq \overline{\mu}({\gamma}) - 3 = m - 3$$
Therefore this stratum is mapped under the evaluation map to a 
chain of dimension $\overline{\mu}({\beta_1})+2 \leq m - 1$, which is of codimension
at least $3$. The only codimension-$1$ strata must have $u_1$
an index-$1$ holomorphic cylinder between $\gamma$ and $\gamma'$, and $u_2$ a holomorphic plane asymptotic to $\gamma'$ with
one marked point. By the definition of the cylindrical contact homology differential, these strata cancel out 
for a cycle in cylindrical contact homology (which is well defined by Theorem \ref{-101}). Therefore $\sum c_i\, \ev(\mathcal{M}_{\gamma_i})$ is a cycle
in $H_{m+2}(M,\partial M)$. Furthermore since the other boundary strata are of codimension at least $3$, the homology
class of this cycle is invariant under homotopies of $J$, $\alpha$, and all other choices.
Hence $\sum_{i=1}^{k}c_{i}\int_{\mathcal{M}_{\gamma_i}}\ev^*(\theta)$ is an invariant.
\end{proof}

\begin{rmk}
\emph{
This condition on the minimum index of Reeb orbits will be satisfied for the contact forms we consider for a subcritical Stein filling.
}
\end{rmk}

To understand gravitational descendants in SFT, let us first recall their definition in Gromov--Witten theory. 
Let $\mathcal{M}_{g,n}$ denote the compactified moduli space of nodal holomorphic 
maps $u$ from a genus $g$
Riemann surface with $n$ internal marked points, $(\Sigma; p_1,\dots, p_n)$, 
to a closed symplectic manifold $(M,\omega)$ with a compatible almost complex structure $J$.
At each element
$(u,\Sigma; p_1,\dots, p_n)$ of $\mathcal{M}_{g,n}$, the cotangent space to $\Sigma$ at the point $p_i$ is
a complex line, they patch together to form a line bundle $L_i$ over $\mathcal{M}_{g,n}$,
called the $i$-th tautological line bundle. Denote its first Chern class by $\psi_i = c_1(L_i)$.

There is a more geometrical interpretation of $L_i$. Consider the embedding 
$f_i\colon \mathcal{M}_{g,n}\rightarrow \mathcal{M}_{g,n+1}$, adding a 
ghost bubble at $p_i$. In other words, the domain $\widetilde{\Sigma}$ of
$f_i(\Sigma; p_1,\dots, p_n, u)$ is the nodal curve $\Sigma\cup \mathbb{C}P^1$, where $\mathbb{C}P^1$
is attached at $p_i$; the marked points on $\widetilde{\Sigma}$ are located at their original positions
on $\Sigma$, except for $p_i$, which is now on $\mathbb{C}P^1$ together with $p_{n+1}$; the map 
$\tilde{u}$ equals $u$ on $\Sigma$, and is the constant map $u(p_i)$ on $\mathbb{C}P^1$. Then $L_i$ 
is the pullback of the dual of the normal bundle of $\mathcal{M}_{g,n}$ in 
$\mathcal{M}_{g,n+1}$.

Let $\{\theta_i\}_{i=1}^{n}$ be closed forms on $M$ of compact support, then a \emph{gravitational descendant} is an integral of the form 
\begin{equation*}\label{eq:desc}
\int_{\mathcal{M}_{g,n}}\ev_1^*(\theta_1)\wedge \psi_1^{l_1}\wedge\cdots\wedge\ev_n^*(\theta_n)\wedge \psi_n^{l_n}.
\end{equation*}
Its value depends only on the cohomology classes of $\{\theta_i\}_{i=1}^{n}$.

From now on let us restrict to rational curves with one marked point. 
For exactly the same reason as correlators, the values of the gravitational 
descendants in SFT will depend on the actual form $\theta$, as well as all choices 
made in the perturbation theory.

In Gromov--Witten theory, the $\psi$ class can be interpreted as the zero set 
of a generic section $s$ of the tautological line bundle. We can do the same
in the SFT setting. The $\psi$ class can be interpreted as the zero sets of 
a collection of generic coherent sections $\{s(\mathcal{M}_{\gamma})\}$ 
over all moduli spaces $\mathcal{M}_{\gamma}$, 
which is compatible with the restriction maps to boundary strata. 
For example, suppose a boundary stratum of $\mathcal{M}_{\gamma}$
is of the form $\mathcal{M}_{\gamma; \gamma'}\times \mathcal{M}_{\gamma'}$, 
then $s(\mathcal{M}_{\gamma})$ restricted 
to $\mathcal{M}_{\gamma; \gamma'}\times \mathcal{M}_{\gamma'}$ is the pullback of $s(\mathcal{M}_{\gamma'})$ under
the projection map $\mathcal{M}_{\gamma; \gamma'}\times \mathcal{M}_{\gamma'} \rightarrow \mathcal{M}_{\gamma'}$.
Higher powers $\psi^{l}$ can be inductively defined as the zero 
sets of generic coherent sections of $L^{\otimes l}$ over the zero sets representing $\psi^{l-1}$, weighted by 
a factor of $\frac{1}{l}$, since $c_1(L)=\frac{1}{l}c_1(L^{\otimes l})$.
For a detailed treatment of coherent sections, see \cite{Fabert}. 
Note that coherent collections can always be constructed inductively. 

Similar to Proposition \ref{invariantcor}, by dimension count we see that certain linear
combinations of descendants are numerical invariants of the symplectic filling. The proof
is almost identical, just replace all moduli spaces by the zero sets of suitable coherent 
sections.

\begin{prop}
Let $(M, \partial M)$ be a symplectic filling with a contact form $\alpha$ on $\partial M$ such that all
Reeb orbits have index at least $2$. If $a = [\sum_{i=1}^{k} c_{i}\gamma_{i}] \in HC_{m+2l}(\partial M)$ is a cycle
in cylindrical contact homology, and $\theta$ is a compactly supported closed $(m+2)$-form, then the value of the linear 
combination of descendants $$\sum_{i=1}^{k}c_{i}\int_{\mathcal{M}_{\gamma_i}}\ev^*(\theta)\wedge \psi^{l}$$
is independent of all choices. 
\end{prop}

As pointed out in \cite{Pandharipande}, gravitational descendants are very closely related to 
counting curves with certain ramification conditions at the marked point. 

Let $(M^{2n},\partial M,J)$ be an exact symplectic filling with a compatible almost
complex structure $J$. Suppose that, for the sake of simplicity, there is a point 
$p\in M$ such that $J$ is standard in a neighbourhood of $p$. 
Let $\M_{\gamma}(p)$ denote the moduli 
space of holomorphic planes with one marked point asymptotic to a 
Reeb orbit $\gamma$ on $\partial M$, such that the marked point is mapped 
to $p$. Choose a neighbourhood, $B_p$ of $p$, biholomorphic
to the unit ball in $\C^n$, and a complex direction $\C\subset \C^n$. Let $\pi$ be
the orthogonal projection from $B_p$ to $\C$.

The domain of the holomorphic curves, $\C$ with one marked point, is not stable. We can normalize so
that the marked point is $0$. The automorphism group is exactly $\C$. The addition of an extra marked 
point makes the domain stable. We will call the space of holomorphic planes with one marked 
point the \emph{unparametrized} curves, and the space of holomorphic planes with two marked points
the \emph{parametrized} curves.
The tautological line bundle $L$ over $\M_{\gamma}(p)$ is then the dual of the bundle of 
parametrized curves over unparametrized curves. The fiber over each element $u$ of
$\M_{\gamma}(p)$ consists of a holomorphic map $\tilde{u}\colon \C \rightarrow M$, together with the
$\C$-family of reparametrizations of $\tilde{u}$, $\{\tilde{u}(cz),c\in \C\}$. Note that $c=0$ corresponds
to the nodal curve $\C\cup \C P^1\rightarrow M$, where a constant ghost bubble is attached to $u$.

Observe that $\tilde{u} \rightarrow \frac{\partial}{\partial z}|_{z=0}(\pi\circ \tilde{u})$ is a section of the dual
of the tautological line bundle. Furthermore such sections form a coherent collection over different $\M_{\gamma}(p)$'s. The
zero set of this section consists of (unparametrized) holomorphic curves $u$ whose representative $\tilde{u}$,
after projection onto the chose $\C$ direction, has the form $z\rightarrow cz^{k}$ for some $k\geq 2$.
We will often suppress the chosen complex direction, and simply refer to this zero set as curves with 
ramification index $2$.

Similarly, over the curves with ramification index $2$, 
$\tilde{u}\rightarrow \frac{\partial^2}{\partial z^2}|_{z=0}(\pi\circ \tilde{u})$ is a section 
of $L^{\otimes 2}$. The zero set of this section is referred to as curves with ramification index $3$.

Therefore if $\theta$ is Poincar\'{e} dual to the point class, then 
$\int_{\mathcal{M}_{\gamma}}\ev^*(\theta)\wedge \psi^{l}$ can be interpreted as the count of holomorphic
planes passing through $p$ with ramification index $(l+1)$, divided by $l!$.

\begin{rmk}\label{ramcond}
\emph{The above construction can be made to work in more general settings. 
However for simplicity the following is sufficient for the purpose of this paper.
Suppose there exists an open set $L\subset M$ satisfying the following:
\begin{itemize}
\item $L$ is biholomorphic to a submanifold of a product almost complex manifold 
$M'\times \C$;
\item for each element of $H(M)$, there exists a cycle representative which lies entirely in $L$.
\end{itemize}
Then ramification index can be defined with respect to the vertical complex 
direction $\C$ as above, so that 
$\int_{\mathcal{M}_{\gamma}} (l!) \ \ev^*(\theta)\wedge \psi^{l}$ is the count of holomorphic 
planes passing through a cycle $\alpha\subset L$ Poincar\'{e} dual 
to $\theta$, with ramification index $(l+1)$. 
}
\end{rmk}

\section{Stein and Weinstein manifolds}

An open complex manifold $(M^{2n},J)$ is \emph{Stein} if
it can be realized as a properly embedded complex submanifold
of some $\mathbb{C}^{N}$. 
A smooth function $f\colon M\rightarrow \mathbb{R}$ is \emph{exhausting}
if it is proper and bounded from below. Let $d^{J}f$ denote
$df\circ J$. The function $f$ is
\emph{plurisubharmonic} if the associated 2-form $\omega_{f} = -dd^{J}f$
is a symplectic form taming $J$, i.e., $\omega_{f}(v,Jv)>0$ for every
non-zero tangent vector $v$. Plurisubharmonicity is an open condition.
We can therefore assume $f$ to be Morse. By a theorem of Grauert, an
open complex manifold is Stein if and only if it admits a plurisubharmonic
function.

A Stein manifold $(M^{2n},J)$ with an exhausting plurisubharmonic
function $f$ admits the following associated structures:

\begin{itemize}
\item a symplectic form $\omega_{f} = -dd^{J}f$ which is $J$-invariant,
\item a primitive $\alpha = -d^{J}f$,
\item a vector field $Y$ such that $\alpha=\iota_{Y}\omega$,
\item a metric $g(v,w)=\omega(v,Jw)$.
\end{itemize}

Since $L_{Y}\omega = \iota_{Y}d\omega + d(\iota_{Y}\omega) = d\alpha =\omega$,
the vector field $Y$ is Liouville, i.e., the flow of $Y$ expands the symplectic
form. In fact $Y$ is the gradient vector field of $f$ with respect to the metric $g$,
$$df=-df\circ J \circ J=\alpha \circ J= (\iota_{Y}\omega) \circ J=\iota_{Y}g.$$

The function $f$ can always be rescaled so that $Y$ becomes a complete 
vector field. We will assume all plurisubharmonic functions
produce complete Liouville vector fields unless otherwise specified. If
$M$ admits a plurisubharmonic Morse function $f$ with finitely many
critical points, then $M$ is of \emph{finite type}. The unstable
submanifold of each critical point of $f$ is an isotropic
submanifold with respect to the symplectic form $\omega_{f}$, so the
Morse index is no greater than the complex dimension of $M$. A Stein
manifold $M$ is \emph{subcritical} if it admits a plurisubharmonic
Morse function $f$ with all critical points having Morse index
strictly less than $n$.

From now on we will restrict ourselves to Stein manifolds of finite type. All plurisubharmonic
functions are assumed to be Morse unless otherwise stated. By a theorem of Eliashberg and Gromov,
a Stein manifold carries a canonical symplectic structure. A different choice of plurisubharmonic
function corresponds to a different choice of complete Liouville vector field $Y$ on the same
symplectic manifold.

\begin{thm}[\cite{EliashbergGromov}]
Let $f$ and $g$ be two plurisubharmonic functions which induce complete
Liouville vector fields on a Stein manifold $M$. Then the manifolds
$(M, \omega_{f})$ and $(M, \omega_{g})$ are symplectomorphic.
\end{thm}

Two Stein structures $(M, J_{0})$ and $(M, J_{1})$ are
\emph{Stein homotopic} if there is a continuous family
of Stein structures $(M, J_{t})$ with exhausting
plurisubharmonic functions $f_{t}$ such that
the critical points of $f_{t}$ stay in some compact
subset during the homotopy. Two Stein manifolds
$(M_{0}, J_{0})$ and $(M_{1}, J_{1})$
are \emph{deformation equivalent} if there exists
a diffeomorphism $\phi\colon M_{0} \rightarrow M_{1}$
such that $(M_{0},J_{0})$ and $(M_{0},\phi^{*}J_{1})$
are Stein homotopic.

A Stein manifold is \emph{split} if it is of the form $(M'\times
\mathbb{C},J'\times i)$, where $(M',J')$ is Stein. For a split Stein
manifold we can use the plurisubharmonic function 
$f=f' + \frac{1}{4}(x_n^2+y_n^2)$, where $f'$ is a plurisubharmonic
function on $M'$ and $(x_n,y_n)$ is the Euclidean coordinate of
$\mathbb{C}$.

\begin{thm}[\cite{Cieliebak}]
Every subcritical Stein manifold is deformation
equivalent to a split one.
\end{thm}

If $(M_{0},J_{0})$ and $(M_{1},J_{1})$ are
deformation equivalent, then $(M_{0},\omega_{f_0})$ and
$(M_{1},\omega_{f_1})$ are symplectomorphic.
Hence we will always treat a subcritical Stein manifold
as split.

It turns out that the existence of a integrable complex structure on $M$ 
is purely topological.

\begin{thm}[\cite{EliashbergStein}]\label{steinstructure}
For $n\geq 3$, any almost complex structure $J$ on a $2n$-dimensional manifold $M$
which admits an exhausting Morse function $f$ with all critical points
of index $\leq n$ is homotopic to an integrable Stein complex structure
$\tilde{J}$ such that $f$ is $\tilde{J}$-plurisubharmonic.
\end{thm}

We may then work with the more relaxed notion of a \emph{Weinstein} manifold. 
A symplectic manifold $M^{2n}$ is \emph{Weinstein} if it admits a Liouville vector field
$Y$ and a Morse function $f$ such that $Y$ is gradient-like with respect to $f$.
It is of \emph{finite type} if $f$ has finitely many critical points, \emph{complete} if the 
Liouville vector field is complete,
and \emph{subcritical} if all critical points are of index strictly less than $n$. For a complete Weinstein
manifold of finite type, we can normalize $f$ so that all critical values are less
than $1$. Then $M$ is a symplectic filling with a cylindrical end
symplectomorphic to $V\times [0,\infty)$, where $V$ is the level set $\{f = 1\}$.
The closed subset $\{f \leq 1 \}$ is called a \emph{Weinstein domain}. We will often
abuse notation and denote the level set $V$ by $\partial M$, and the Weinstein domain by $M$ as well.
From now on all Weinstein manifolds are assumed complete and of finite type.

Similarly, a Weinstein structure $(M, \omega, Y, f)$ is \emph{split} if it is of the form
$$(M'\times \mathbb{C},\  \omega' + dx\wedge dy, \ Y'+\frac{1}{2}(x\del{x} + y\del{y}), \ f' + \kappa(x^2+y^2)),$$ 
where $(M', \omega', Y', f')$ is Weinstein, $(x,y)$ is the Euclidean coordinate on $\mathbb{C}$, and $\kappa$
a positive real constant. 

A Weinstein manifold can be reconstructed by 
symplectic handle attachments in the same way as classical Morse theory.

An index-$k$ handle of real dimension $2n$ is
modeled on the complex $n$--dimensional space $\mathbb{C}^n$ with
the standard symplectic form $\omega_{\st}$ together with a
standard complete Liouville vector field $Y_{\st}$.
Let $(x_i,y_i)$ be the Euclidean coordinates,
$\omega_{\st}=\sum _{i=1}^n dx_i\wedge dy_i$.

Define
\[
Y_{\st} =\sum_{i=1}^{k}\left(2x_i\frac{\partial}{\partial
x_i}- y_i\frac{\partial}{\partial  y_i}\right)+\sum _{j=k+1}^n\frac{1}{2}\left(
x_j\frac{\partial}{\partial x_j}+y_j\frac{\partial}{\partial y_j}\right),
\]
$Y_{\st}$ is the gradient vector field of the function
$$f_{\st}=\sum_{i=1}^{k}(x_i^2- \frac{1}{2}y_i^2)+
\sum_{j=k+1}^{n}\frac{1}{4}(x_j^2+y_j^2)$$ with respect to the
Euclidean metric. It is easy to check that $L_{Y_{\st}}\omega_{\st}=\omega_{\st}$.

Let the
$1$-form $\alpha_{\st}$ be the contraction $\iota_{Y_{\st}}\omega_{\st}$, it restricts
to a contact $1$-form on any hypersurface $V$ transverse to
$Y_{\st}$.
$$\alpha_{\st}=\sum_{i=1}^{k}(2x_idy_i + y_idx_i)+
\sum_{j=k+1}^{n}\frac{1}{2}(x_jdy_j - y_jdx_j).$$

The \emph{core} of an index-$k$ handle is an isotropic $k$-disk
$$D_{\st}=\left\{\sum_{i=1}^{k}y_i^2 \leq 1, x_i=x_j=y_j=0\right\}.$$
Let its boundary sphere be
$$S_{\st}= \left\{\sum_{i=1}^{k}y_i^2 = 1, x_i=x_j=y_j=0\right\}.$$

Let $b_1, b'_1, \dots, b_k, b'_k, a_{k+1}, \dots, a_n$ be positive constants. 
A \emph{standard contact handle} is the surface
$$V_+=\left\{\sum_{i=1}^{k}(b_i x_{i}^2 - b'_i y_{i}^2) + \sum_{j=k+1}^n a_{j}(x_{j}^2+y_{j}^2) = 1\right\}.$$
The standard contact handle $V_+$ is everywhere transverse to the Liouville vector field $Y_{\st}$.
Let $U$ to be the interior of $V_+$,
$$U=\left\{\sum_{i=1}^{k}(b_i x_{i}^2 - b'_i y_{i}^2) + \sum_{j=k+1}^n a_{j}(x_{j}^2+y_{j}^2) \leq 1\right\}.$$
For any set of constants $\{b'_i\}$, we can choose $\{b_i, a_j\}$ sufficiently large such that $U$ is an arbitrarily small tubular
neighbourhood of $D_{\st}$.
The cylinder $V_{-}=\{\sum_{i=1}^{k}y_i^2 = 1\}$ is transverse to the
Liouville vector field $Y_{\st}$, and $S_{\st}$ is an isotropic 
sphere on the contact manifold $V_{-}$. 

Suppose we have a Weinstein domain 
$(M,\omega,Y,\partial M)$. If $S$ is an isotropic 
$(k-1)$-sphere on
$\partial M$ together with a trivialization of the
symplectic subnormal bundle, then using
the following standard neighbourhood theorem,
we can attach a sufficiently small tubular
neighbourhood $U$ of $D_{\st}$ to $M$, identifying
$S_{\st}$ with $S$:

\begin{thm}[\cite{Weinstein}]\label{wei}
For $i=1,2$ let $(M_i,\omega_i,Y_i,V_i,S_i)$ be a symplectic manifold
$(M_i,\omega_i)$ with Liouville vector field $Y_i$, hypersurface
$V_i$ transverse to $Y_i$, and isotropic submanifold $S_i$ of
$V_i$. Given a diffeomorphism from $S_1$ to $S_2$ covered by an
isomorphism between their symplectic subnormal bundles, there exist
neighbourhoods $N_i$ of $S_i$ in $M_i$ and a symplectomorphism $\Phi$
between them extending the given diffeomorphism between $S_i$, such
that $(N_1,\omega_1,Y_1,V_1,S_1)$ is taken to
$(N_2,\omega_2,Y_2,V_2,S_2)$.
\end{thm}

Apply Theorem \ref{wei} to the pair $(M,\omega,Y,\partial M,S)$
and $(\mathbb{C}^n,\omega_{\st},Y_{\st},V_{-},S_{\st})$ to attach 
$U$ to $M$, using
the natural framing for the trivial symplectic normal bundle of $S_{\st}$,
$$\left\{\del{x_{k+1}},\del{y_{k+1}},\dots, \del{x_n},\del{y_n}\right\}.$$

The new Weinstein domain $M\cup U$ is called a \emph{symplectic handle attachment of index-$k$}.
The Liouville vector field $\widetilde{Y}$ on $M\cup U$ is the original $Y$ on $M$ and $Y_{\st}$ on $U$. 
The boundary of $M\cup U$ is called a \emph{contact handle attachment of index-$k$}. Topologically, 
$\partial (M\cup U)$ is the result of an index-$k$ surgery on $\partial M$. 

Note that after a handle is attach, the same attaching map $\Phi$ can be used 
to attach a thinner handle $U'\subset U$. 

\begin{rmk}\label{extendf}
\emph{
The Morse function $f$ on $M$ can also be extended to $M\cup U$. Define $g$ on $U$ to be the quadratic
$$g = \frac{1}{K}\left(\sum_{i=1}^{k}(b_i x_{i}^2 - b'_i y_{i}^2) + \sum_{j=k+1}^n a_{j}(x_{j}^2+y_{j}^2)\right)+\frac{K-1}{K}$$
The level set $\{g = 1\}$ is exactly $V_+$. For sufficiently large $K$, $\tilde{f}=\min(f,g)$ is an extension of $f$ to $M\cup U$.
}
\end{rmk}

\begin{rmk}\label{rmk1}
\emph{
The contact handle attachment $\partial (M\cup U)$ has a corner where $V_+$ intersects $V_-$. In order
to have a smooth surface we have to round the corner. The rounding procedure can be done in an 
arbitrarily small neighbourhood of the intersection, and will not affect any of our arguments. In terms of the Morse function $\tilde{f}$,
this means we have to smooth $\tilde{f}$ in near where $f=g$. Again this can be done in an arbitrarily small neighbourhood of $V_{-}$
(the set $\{f=g\}$ can be brought into any small neighbourhood of $V_-$ by choosing $K$ large), and will not
affect any argument. Hence from now on we will assume that the Morse function $f$ for a Weinstein domain $M$ is quadratic on each handle.
}
\end{rmk}

Let $M^{2n} = (M'\times \mathbb{C},\  \omega' + dx_n \wedge dy_n, \ Y'+\frac{1}{2}(x_n\del{x_n} + y_n\del{y_n}), \ f' + \kappa(x_n^2+y_n^2))$ 
be a split Weinstein manifold of finite type such that $c_1(M) = 0$. By Remark \ref{rmk1} we can
assume that $(M',f')$ is obtained from handle attachments, and that $f'$ is quadratic on each handle. 
Then clearly by construction, $M$ is made up from handles of the same Morse index, 
and $f$ is also quadratic on each handle. The handle attaching maps of $M$ are precisely those of $M'$ 
together with the identity map on the $\mathbb{C}$-factor. Let $\rho$ be the 
vertical projection $\partial M\rightarrow M'$. Note that the part of $\partial M$ over 
$M'\setminus \partial M'$ is a circle bundle.

\begin{prop}\label{prop1} 
Let Let $M^{2n} = (M'\times \mathbb{C},\  \omega' + dx_n \wedge dy_n, \ Y'+\frac{1}{2}(x_n\del{x_n} + y_n\del{y_n}), \ f' + \kappa(x_n^2+y_n^2))$ 
be a split Weinstein domain of finite type such that $c_1(M) = 0$.
Given any positive integer $N$, there exists $\kappa$ sufficiently large such that any Reeb orbit on $\partial M$ 
of index $\leq N$ is a multiple cover of a circle fibre $\rho^{-1}(p)$ over a critical point $p$ of $f'$ on $M'$. 
Denote this orbit by $\gamma_{p}^{m}$, where $m$ is its multiplicity. The index
of $\gamma_{p}^{m}$ is $2m+2n-4-\indexof(p)$, where $\indexof(p)$ is the Morse index of the critical point $p$.
\end{prop}
\begin{proof}
This is proved in \cite{Yau}. We sketch an argument:
the Reeb flow is a bounded multiple of the Hamiltonian flow, hence they share the same orbits and the Maslov index for the Hamiltonian flow can be 
used to bound the Maslov index of the Reeb flow. The Hamiltonian flow on $\partial M$ consists of a constant rotation in 
the $\mathbb{C}$-component and a Hamiltonian flow on $M'$. As we increase $\kappa$, the rotational speed in the $\mathbb{C}$-component increases while the Hamiltonian flow on $M'$ remains the same. 
If the image of an orbit $\gamma$ is not constant on $M'$, 
then for $\gamma$ the number of complete rotations in the $\mathbb{C}$-component will also increase with $\kappa$, 
adding $2$ to the Maslov index per rotation.
Therefore the the only orbits whose index stay bounded are those that project to a constant on $M'$. We refer to Theorem $3.1$ and Lemma $4.4$ of \cite{Yau} for details and the index computation.
\end{proof}

\begin{rmk}\label{minindex}
\emph{Since $M$ is subcritical, $\indexof(p)\leq n-1$, hence the minimum index of a Reeb orbit is $2+2n -4-(n-1)=n-1\geq 2$. The index calculation also shows that both the cylindrical
and contact homology algebra exist only in the trivial class of the homology grading in Remark \ref{homologygrade}, 
since the orbits $\gamma_{p}^m$ are in fact contractible on $\partial M$}.
\end{rmk}

Rotation in the $\mathbb{C}$-factor, $(p,z)\rightarrow (p,e^{i\theta}z)$, is an $S^{1}$-action on the split Weinstein domain $(M,\partial M)$. We
can choose an $S^1$-invariant compatible complex structure $J$ on the contact distribution $\xi$ of $\partial M$. Denote $\del{\theta}=x_n\del{y_n}-y_n\del{x_n}\in T\mathbb{C}\subset TM$ to 
be the vector field generating the $S^1$-rotation, and $\del{r} = x_n\del{x_n} + y_n\del{y_n}$. 
Let the associated vector field $Z = d\rho (J\del{\theta}|_{\partial M})$ be the projection of $J\del{\theta}$ onto $TM'$.
Suppose $u(s,t)\colon \mathbb{R}\times S^1 \rightarrow \partial M\times \mathbb{R}$ 
is an $S^1$-invariant holomorphic cylinder in the 
symplectization $\partial M\times \mathbb{R}$, then up to $\mathbb{R}$-translation, $u$ is
uniquely determine by the projection $u'(s,t)\colon \mathbb{R}\times S^1 \rightarrow \partial M$.
Since $u'(s,t)$ is still $S^1$-invariant, it projects to a trajectory of $Z$ on $M'$. Conversely, for each multiplicity $m$,
a trajectory of $Z$ on $M'$ lifts to a $m$-fold covered holomorphic cylinder $u$, 
such that $\deldel{u}{t} = m\del{\theta}$ and $\deldel{u}{s} = -J\del{\theta}$.

\begin{prop}\label{regulargradient}
If the vector field $Z$ is of Morse--Smale type, then each $S^1$-invariant holomorphic cylinder 
$$u(s,t)\colon \mathbb{R}\times S^1 \rightarrow \partial M\times \mathbb{R}$$ is regular.
\end{prop}
\begin{proof}
Lemma $7.5$ of \cite{Yau} and Theorem $7.3$ of \cite{Salamon}.
\end{proof}

\begin{rmk}\label{rmk2}
\emph{We will define $J$ such that it takes a standard form near the critical points, so $Z$ is Morse. A generic perturbation of $J$ 
away from the critical points will then make $Z$ Morse--Smale.}
\end{rmk}

Consider the $k$-fold branch cover 
$$\Theta_k\colon \partial M \rightarrow \partial M,\ \  \Theta_k(p, re^{i\theta}) = (p,re^{ik\theta}).$$
The contact form on $\partial M$ takes the form 
$$\alpha = \iota_{Y}\omega = \iota_{Y'}{\omega'}+\frac{1}{2}(x_ndy_n - y_ndx_n).$$
Since $\alpha$ is $S^1$-invariant, we can push forward $\alpha$ to another contact form $\alpha_{k}$
on $\partial M$, $$\alpha_{k} = (\Theta_{k})_*\alpha = \iota_{Y'}{\omega'}+\frac{1}{2k}(x_ndy_n - y_ndx_n).$$
The push forward of an $S^1$-invariant compatible complex structure $J$ on $(\partial M,\alpha)$ will be an
$S^1$-invariant compatible complex structure on $(\partial M,\alpha_k)$,
$$J_k = d\Theta_{k} \circ J \circ d\Theta_{k}^{-1}.$$ 

\begin{rmk}\label{rmk3}
\emph{Recall that $(\partial M, \alpha)$ is identified with the hypersurface $V=\{f'+ \kappa (x_n^2+y_n^2) = 1\}$ in $M$. 
We can identify $(\partial M, \alpha_k)$ with the hypersurface $V_k=\{f'+ k\kappa (x_n^2+y_n^2)=1\}$ as follows:
let $\Psi_k$ be the diffeomorphism $M'\times\mathbb{C} \rightarrow M'\times\mathbb{C}\colon \Psi_k(p,x_n+iy_n) = 
(p,\frac{1}{\sqrt{k}}(x_n+iy_n))$. Then $\Psi_{k}(V)=V_{k}$, and
$\alpha_k = \Psi_{k}^*(\iota_{Y}\omega|_{V_k})$ 
is the pullback of the restriction of $\iota_{Y}\omega$ to $V_k$.
We use $\Psi_k$ to push forward the complex structure $J_k$ on $V$ to $V_k$, and will still denote it by $J_k$. The associated vector 
field $Z_{k}$ for $(V_{k},J_{k})$ is just $\frac{Z}{k}$. In particular, it stays Morse--Smale.
}
\end{rmk}

\begin{prop}\label{prop2}
Given an $S^1$-invariant compatible complex structure $J$ on $(\partial M,\alpha)$ such that the associated
vector field $Z$ is Morse--Smale, there exist $k_0$ such that for all $k > k_0$, and all pairs of Reeb orbits
$(\gamma, \gamma')$ with $\overline{\mu}({\gamma'})+1=\overline{\mu}({\gamma})< N$, 
the elements of $\mathcal{M}(\gamma, \gamma')$ in the symplectization of 
$(\partial M, \alpha_k, J_k)$ are $S^1$-invariant.
\end{prop}
\begin{proof}
Lemma 7.6 of \cite{Yau}.
\end{proof}

Theorem \ref{yau1} follows immediately from Propositions \ref{regulargradient} and \ref{prop2}. 
Moreover, the associated vector field $Z$ is always gradient--like for the function $f'$. Thus 
the differential for $HC(\partial M)$ for each orbit multiplicity coincides with the Morse 
differential of $f'$ on chain level.

\begin{lem}\label{gradlike}
The vector field $Z$ is gradient--like with respect to $f'$.
\end{lem} 
\begin{proof}
Since $\del{\theta}$ is tangent to $\partial M$, in the decomposition
$TM|_{\partial M}=\xi \oplus \mathbb{R}Y \oplus \mathbb{R}R$, we have 
\begin{eqnarray*}
\del{\theta} &=& \del{\theta}- \alpha\left( \del{\theta} \right) \cdot R 
\oplus 0\cdot Y \oplus \alpha \left( \del{\theta} \right) \cdot R\\ \label{eqn1}
&=& \del{\theta}- \frac{1}{2}(x_n^2+y_n^2)\cdot R 
\oplus 0\cdot Y \oplus \frac{1}{2}(x_n^2+y_n^2)\cdot R \\ 
\label{eqn2}
J\del{\theta} &=& J\left(\del{\theta}- \frac{1}{2}(x_n^2+y_n^2)\cdot R\right) \oplus
-\frac{1}{2}(x_n^2+y_n^2)\cdot Y \oplus 0\cdot R
\end{eqnarray*}

The first projection to $\partial M= V \times \{0\}$ ignores the $Y$-coordinate (the $\mathbb{R}$-coordinate of
the symplectization),
\begin{eqnarray*}
Z &=& d\rho \left( J\left(\del{\theta}- \frac{1}{2}(x_{n}^2+y_{n}^2) R\right) \right) 
\end{eqnarray*}

Since $J\left(\del{\theta}- \frac{1}{2}(x_{n}^2+y_{n}^2) R\right)\in \xi$,
\begin{eqnarray*}
\alpha\left(J\left(\del{\theta}- \frac{1}{2}(x_{n}^2+y_{n}^2) R\right)\right)&=&0 \\
df\left(J\left(\del{\theta}- \frac{1}{2}(x_{n}^2+y_{n}^2) R\right)\right)&=&0
\end{eqnarray*}

From these we can find the coefficient of $\del{r}$ and $\del{\theta}$,
$$J\left(\del{\theta}- \frac{1}{2}(x_{n}^2+y_{n}^2) R\right)= Z - 
\frac{2\alpha'(Z)}{x_{n}^2+y_{n}^2} \del{\theta} - 
\frac{df'(Z)}{2\kappa(x_n^2+y_n^2)} \del{r} $$

Since $\omega = \omega' + dx_n\wedge dy_n$ tames $J$, if $\del{\theta}\neq \frac{1}{2}(x_{n}^2+y_{n}^2) R$
(i.e., we are not over a critical point of $f'$), then
\begin{eqnarray*}
0&<& \omega \left(\del{\theta}- \frac{1}{2}(x_{n}^2+y_{n}^2) R, \ \  
Z - \frac{2\alpha'(Z)}{x_{n}^2+y_{n}^2} \del{\theta} - 
\frac{df'(Z)}{2\kappa(x_n^2+y_n^2)} \del{r} \right)\\
&=& \omega \left(\del{\theta}, \ \  
Z - \frac{2\alpha'(Z)}{x_{n}^2+y_{n}^2} \del{\theta} - 
\frac{df'(Z)}{2\kappa(x_n^2+y_n^2)} \del{r} \right)\\
&=& \frac{df'(Z)}{2\kappa}
\end{eqnarray*}

Thus $df'(Z)>0$ for $Z\neq 0$, and $Z$ is indeed gradient--like.
\end{proof}

\begin{rmk}
\emph{
In standard Morse theory the \emph{negative} gradient flow is used to define stable and unstable submanifolds of a critical point. We
keep this convention: the unstable submanifold of $p$, $U_{p}$, consists of positive trajectories of $-Y$ flowing
out of $p$.} 
\end{rmk}

The new ingredient in this paper is to construct a suitable extension of a compatible $S^1$-invariant complex structure 
$J$ on $\partial M$ to the entire symplectic filling. The technical result, which will be proved in Section 5, 
is the following: 

\begin{prop}\label{tech}
Let $M^{2n} = (M'\times \mathbb{C},\  \omega' + dx_n \wedge dy_n, \ Y'+\frac{1}{2}(\del{x_n} + \del{y_n}), \ f' + \kappa(x_n^2+y_n^2))$ 
be a split Weinstein manifold of finite type such that $c_1(M) = 0$, and $\rho\colon M\rightarrow M'$ the projection onto first factor.
Given any positive integer $N$,  
there exists a modification of $f'$, $\kappa$ sufficiently large, and a compatible complex structure 
$J$ on $M$ such that the following hold:
\begin{enumerate}
\item
the complex structure $J$ is $S^1$-invariant. Moreover $J = J|_{TM'} \oplus i$ on 
$M' = M'\times \{0\}$, where $i$ is the standard complex structure on $\C$, so
$M'$ becomes an almost complex submanifold;
\item 
the contact form $\alpha$ and compatible complex structure $J$ on $\partial M$ satisfy
\emph{Proposition \ref{prop1}} and \emph{Proposition \ref{prop2}};
\item
$J$ satisfies \emph{Remark \ref{ramcond}}, there is a neighbourhood $L$ of the union of
the unstable submanifolds $U_{p}$ on which $J(p,z) = J'(p)\oplus i$ is split;
\item
if $p$ and $q$ are critical points of $f'$ and $\indexof(q)< \indexof(p)$, then there is no
holomorphic plane asymptotic to a Reeb orbit $\gamma_{p}^{m}$, 
$\overline{\mu}(\gamma_{p}^{m})\leq N$, which passes through a point on $U_{q}$;
\item
if $\indexof(q) = \indexof(p)$, then there is precisely one holomorphic plane asymptotic to 
$\gamma_{p}^{m}$, such that the marked point is mapped to $U_{q}$ with ramification index $m$, 
namely the $m$-fold branch cover of the vertical plane $\rho^{-1}(p)$ over $p$, $\rho^{-1}(p)(z^m)$
(this forces $p=q$). Moreover this plane is regular.
\end{enumerate}
\end{prop}

Assuming Proposition \ref{tech} we will prove Theorems \ref{contacthomology} and \ref{main}.

\begin{proof}[\emph{\textbf{Proof of Theorem \ref{contacthomology}}}]
It is enough to prove Theorem \ref{contacthomology} up to some degree $N_0$ for a suitable choice of contact form $\alpha_0$.
The general statement then follows by taking a sequence $N_i \rightarrow \infty$, and taking a limit 
of the contact homology algebras for $\alpha_i$, using the natural isomorphisms between homology
algebras of different contacts forms.

Let $\Crit(f')$ be the vector space generated by the critical points of $f'$ with the cohomology grading, 
i.e., $p$ has grading $2n-\indexof(p)$. Let $$\Lambda\left(\bigoplus_{m=1}^{\infty} \Crit(f')[2m-4]\right)$$ 
be the graded commutative tensor algebra generated by the directs sum of shifted copies of $\Crit(f')$. Take the
Morse cohomology differential on each copy $\Crit(f')[2m-4]$ and extend by the Leibniz rule. Denote by
$d_{\Morse}$ the resultant differential on  $\Lambda\left(\bigoplus_{m=1}^{\infty} \Crit(f')[2m-4]\right)$.

The homology of the differential grade algebra 
$\left(\Lambda\left(\bigoplus_{m=1}^{\infty} \Crit(f')[2m-4]\right), d_{\Morse}\right)$ is easily seen to be
$\Lambda\left(\bigoplus_{m=1}^{\infty} H(M,\partial M)[2m-4]\right)$.

Fix some $N_0$. For a suitable contact form $\alpha_0$, we will construct a chain map $\Phi$ between the 
differential graded algebra chain complex for contact homology with respect to $\alpha_0$, and the complex
$\left(\Lambda\left(\bigoplus_{m=1}^{\infty} \Crit(f')[2m-4]\right), d_{\Morse}\right)$. We finish the 
proof by showing that $\Phi$ is a quasi-isomorphism in degrees $< N_0$.

Choose a setup $(M,\partial M, J)$ as in Proposition \ref{tech}. From now on we will only consider the
truncated complexes with grading $\leq N_0$. Denote the critical point $p \in \Crit(f')[2m-4]$ by $p^m$.
Choose a total order on the critical points, such that $p < q$ if $2n-\indexof(p)<2n-\indexof(q)$.

Write down each generator of $\Lambda\left(\bigoplus_{m=1}^{\infty} \Crit(f')[2m-2]\right)$ 
in the form $\prod_{i=1}^{k}p_{i}^{m_i}$, such that 
if $i < j$, then $p_i\leq p_j$, and $m_i<m_j$ if $p_i=p_j$. Let 
$|\prod_{i=1}^{k}p_{i}^{m_i}|=\sum_{i=1}^{k}(2n-\indexof(p_i)+2m_i-4)$
denote the grading of $\prod_{i=1}^{k}p_{i}^{m_i}$.
Define a total order on the generators as follows
$$\prod_{i=1}^{k}p_{i}^{m_i} < \prod_{i=1}^{l}q_{i}^{n_i}$$
\begin{enumerate}
\item if $|\prod_{i=1}^{k}p_{i}^{m_i}| < |\prod_{i=1}^{l}q_{i}^{n_i}|;$
\item if the gradings are equal and $k > l$;
\item if the gradings are equal, $k=l$, and 
$(p_1,p_2,\dots, p_k) < (q_1,q_2,\dots, q_k)$ in the lexicographical order induced by the 
total order on the critical points.
\end{enumerate}

As in Proposition \ref{prop1}, the generators of the chain complex for the contact 
homology algebra are of the form $\prod_{i=1}^{k}\gamma_{p_i}^{m_i}$. There is a 
natural bijection between these and the
generators of $\Lambda\left(\bigoplus_{m=1}^{\infty} \Crit(f')[2m-2]\right)$, 
taking $\gamma_p^m$ to $p^m$. Note that this bijection preserves grading. The bijection and 
the order on the generators of $\Lambda\left(\bigoplus_{m=1}^{\infty} \Crit(f')[2m-2]\right)$ induce
an order on the generators of the contact homology complex.

Let $\M_{\gamma}((p_1,m_1),\cdots, (p_k,m_k))$ be the moduli space of holomorphic planes 
with $k$ marked points, $\{x_i\}_{i=1}^k$, asymptotic to $\gamma$, such that each $x_i$ is 
mapped to the unstable manifold $U_{p_i}$ with ramification index $m_i$ (recall that ramification 
index is defined with respect to the vertical complex direction $\C$). The expected dimension of
$\M_{\gamma}((p_1,m_1),\cdots, (p_k,m_k))$ is $\overline{\mu}(\gamma) - |\prod_{i=1}^{k}p_i^{m_i}|$.

Define the map $\Phi$ by
$$\Phi(\gamma_p^m)= \sum_{\overline{\mu}({\gamma_p^m})= |\prod_{i=1}^{k}p_i^{m_i}|} c\prod_{i=1}^{k}p_i^{m_i},$$
where the coefficient $c$ is the count of rigid holomorphic curves in $\M_{\gamma_p^m}((p_1,m_1),\cdots, (p_k,m_k))$,
and the sum is taken over all $k, p_i, m_i$. Extend $\Phi$ to the entire complex by Leibniz rule.

To see that $\Phi$ is a chain map, consider $\gamma_p^m$ and $\prod_{i=1}^{k}p_i^{m_i}$ with 
$|\prod_{i=1}^{k}p_i^{m_i}|=\overline{\mu}({\gamma_p^m})-1$. For each $i$, perturb $U_{p_i}$ by
taking parallel copies very close by. Since $\dim (U_{p_{i}}) \leq n-1$, they can be assume to 
be pairwise disjoint. Since the marked points are constrained on disjoint sets, the complicated
degeneration phenomenon, where marked points collide and ramification indices change, does not occur. 
Therefore the boundary of the one dimensional moduli space 
$\M_{\gamma_p^m}((p_1,m_1),\cdots, (p_k,m_k))$ consists of two types of curves:
\vspace{0mm}
\begin{enumerate}
\item the domain undergoes SFT degeneration: $u$ is a $2$-story curve $(u_1,u_2)$, such that 
$u_1\in \mathcal{M}_{\gamma_p^m; \gamma_1,\dots, \gamma_l}$ is a 
genus-$0$ holomorphic curve in the symplectization of $\partial M$; 
and $u_2$ consists of $l$ holomorphic planes in the filling $M$,
one asymptotic to each $\gamma_j$. The $k$ marked points with constrain and 
ramification condition are distributed on the $l$ planes of $u_2$.
\item one marked point $x_i$ is mapped to the boundary of $U_{p_{i}}$.
\end{enumerate}

Curves of the first type correspond precisely to terms in $\Phi\circ  \partial$, and 
curves of the second type correspond to terms in $d_{\Morse}\circ  \Phi$. Therefore 
$\Phi\circ  \partial + d_{\Morse}\circ  \Phi = 0$, and $\Phi$ is a chain map.

To see that $\Phi$ is a quasi-isomorphism, consider the matrix for
$\Phi$ with respect to the generators ordered as above. By parts (4) and (5) of Proposition \ref{tech},
this matrix is upper triangular, and the diagonal entries are $1$. 
Hence $\Phi$ is invertible and a quasi-isomorphism.
\end{proof}

\begin{proof}[\emph{\textbf{Proof of Theorem \ref{main}}}]
To compute the cylindrical contact homology group in degree $k$, choose $N_0 >> k$ and $m_0$ sufficiently large so that all 
Reeb orbits of index $< N_0$ are of the form $\gamma_{p}^{m}$ with $m\leq m_0$. 
By Proposition \ref{prop2}, for each orbit multiplicity $m$, with a grading shift, the differential for cylindrical contact 
homology is the same as the differential for the Morse cohomology (because we count \emph{positive} flows of a 
gradient--like vector field $Z$) of $M$. Therefore
$$HC_{k}(\partial M) \cong \bigoplus_{m=1}^{\infty} H^{2m+2n-4-k}(M) \cong \bigoplus_{m=0}^{\infty} H_{k+2-2m}(M,\partial M).$$
Define $\iota$ to be this isomorphism. Note that $\iota_m$ detects orbits of multiplicity $(m+1)$.
Furthermore, by Lemma \ref{gradlike}, both vector fields $Z$ and $Y$ are gradient-like with respect to $f$, so they have the same Morse
differential. Therefore the isomorphism $\iota$ is in fact a chain level identification of the generators and differential
of $HC(\partial M)$ with those of the Morse theory of the Liouville vector field $Y$.

Let $\{p_1, p_2,\dots, p_N\}$ be the set of critical points, and
$a = \sum_{i=1}^{N}c_{i}\gamma_{p_{i}}^{m_i}$ an element of $HC_{k}$. The multiplicity
$m_i$ is determined by $2m_i+2n-4 -\indexof(p_{i}) =k$. 
The Poincar\'{e} dual of a compactly supported closed ($k+2-2l$)-form $\theta$ 
can be represented by a pseudo--cycle, $\sum_{j=1}^{N} b_{j}U_{p_{j}}$,
of unstable submanifolds of critical points of index $2n+2l-2-k$, i.e., $b_j=0$ unless
$\indexof(p_j)=2n+2l-2-k$. The descendant 
$\sum_{i=1}^{N} c_{i}\int_{\mathcal{M}_{\gamma_{p_i}^{m_i}}}\ev^*(\theta)\wedge\psi^l$ is the weighted 
algebraic count of the number of holomorphic planes asymptotic to any one of 
$\gamma_{p_{i}}^{m_i}$, and passing through the pseudo--cycle 
$\sum_{i=1}^{N} b_{j}U_{p_{j}}$ with ramification index $(l+1)$. 

If $m_i>l+1$, then the Morse
index of $p_{i}$ is $2m_i+2n-4-k > 2l+2n-2-k$. Therefore by part (4) of Proposition \ref{tech}, 
there is no holomorphic plane asymptotic to $\gamma_{p_{i}}^{m_i}$ and passing through 
$\sum_{i=1}^{N} b_{j}U_{p_{j}}$ (regardless of ramification index). 

If $m_i<l+1$, then by part $1$ of Proposition \ref{tech}, $M'=M'\times \{0\}$ is a real codimension
$2$ complex submanifold. Since the asymptotic orbit $\gamma_{p_i}^{m_i}$ winds $m_i$ times around
$M'$, a holomorphic plane asymptotic to $\gamma_{p_i}^{m_i}$ has intersection number $m_i$ with $M'$.
However a ramification point of index $(l+1)$ already contributes more than $m_i$ to the intersection
number. By positivity of intersections, there is no holomorphic plane asymptotic 
to $\gamma_{p_{i}}^{m_i}$ and passing through $\sum_{i=1}^{N} b_{j}U_{p_{j}}$ with ramification index $(l+1)$. 

If $m_i=l+1$, then by part (5) of Proposition \ref{tech}, 
there is precisely one such regular holomorphic plane whenever $i=j$. Therefore 
$$\sum_{i=1}^{N} c_{i}\int_{\mathcal{M}_{\gamma_{p_i}^{m_i}}}\ev^*(\theta)\wedge\psi^l = \frac{1}{l!}\sum_{i=1}^{N} c_{i}b_i.$$

The natural pairing between homology and cohomology can been seen from 
its Morse theory by representing an element of $H_{2n-k}(M,\partial M) \cong H^{k}(M)$ as a
linear combination of the \emph{stable} submanifolds $S_{p_{i}}$.
The intersection pairing between $H_{2n-k}(M,\partial M)$ and $H_{k}(M)$ is then generated by 
$S_{p_{i}}\cap U_{p_{j}} = \delta_{ij}$. 

Therefore 
$\sum c_{i}b_j \delta_{ij}$ is the intersection number of $\iota_l(a)$, which we realize as 
a linear combination of the stable submanifolds of the critical points, with the Poincar\'{e} 
dual of $\theta$. In other words,
\[\sum_{i=1}^{k}c_{i}\int_{\mathcal{M}_{\gamma_i}}\ev^*(\theta)\wedge \psi^l = \frac{1}{l!}<\iota_l(a),[\theta]>. \qedhere\]
\end{proof}

\section{Subcritical Polarizations}

A \emph{polarized K\"{a}hler manifold}
$(M,\omega,J,\Sigma)$ is a K\"{a}hler manifold 
$(M,\omega,J)$ with an primitive integral K\"{a}hler form $\omega$ and a smooth reduced
complex hypersurface $\Sigma$ representing the Poincar\'{e} dual
of $k[\omega]$. The number $k$ is called the \emph{degree}
of the polarization.

There is a canonical plurisubharmonic function 
$f_{\Sigma}\colon M\setminus \Sigma \rightarrow \mathbb{R}$ associated to 
a polarization. Let $\mathcal{L}$ be the holomorphic
line bundle associated to the divisor $\Sigma$ and $s$ the
holomorphic section (unique up to scalar multiplication) of 
$\mathcal{L}$ whose zero section is $\Sigma$. Choose a hermitian
metric $\| \cdot \|$ such that the compatible metric connection
$\nabla$ has curvature $R = 2\pi i k\omega$. Then define
$$f_{\Sigma}=-\frac{1}{4\pi k}\log\|s(x)\|^{2}.$$

It is not hard to check that $f_{\Sigma}$ is plurisubharmonic
and in fact $-dd^{J}f_{\Sigma} = \omega$. All the critical
points of $f_{\Sigma}$ lie within a compact subset of 
$M \setminus \Sigma$. 

A polarization is \emph{subcritical} if there is a 
plurisubharmonic Morse function $f$ such that
$(M \setminus \Sigma,J,f)$ is a subcritical 
Stein manifold, and $f=f_{\Sigma}$ outside a compact
subset of $M\setminus \Sigma$. 

Split $M$ along a level set $V=\|s(x)\|=\epsilon$ into a composition 
of two symplectic fillings by the standard stretching the
neck construction (described in detail in Section 1.3 of \cite{EliashbergGiventalHofer}). 
One filling is $M\setminus \Sigma$, and the other is the complex normal bundle
of $\Sigma$ in $M$. 

\begin{lem} \label{h1}
If $(M,\Sigma)$ is a subcritical polarization such that
$c_1(M\setminus \Sigma)=0$. Let $D$ be the cohomological 
dimension of $M\setminus \Sigma$, i.e. the highest non-vanishing 
cohomology degree. Then 
\begin{enumerate}
\item
under the $H_1$--grading, the contact homology 
algebra $HC^{\cont}(V)$ vanishes except for the trivial class. 
\item
under the $\mathbb{Z}$--grading, the lowest non-vanishing group of $HC^{\cont}(V)$ is 
$$HC^{\cont}_{2n-2-D}(V) \cong H^{D}(M\setminus \Sigma).$$
\end{enumerate}
\end{lem}
\begin{proof}
By Proposition \ref{tech}, there exist a choice of $V' =\partial (M\setminus \Sigma)$ such
that all non-contractible Reeb orbits have arbitrarily large index. Hence $HC^{\cont}$ is non-trivial
only for the trivial class in $H_1$. As pointed out in Remark \ref{minindex}, each Reeb orbit has
index at least $n-1$, thus a product of more than one orbit has index at least $2n-2$. Hence the chain 
groups of index less than $2n-2$ are linearly generated by orbits $\gamma^{m}_{p}$. Since there
are no products, the contact homology algebra differential coincides with the cylindrical contact homology
differential. By Proposition \ref{prop2}, this coincides with copies of a Morse differential on
$M\setminus \Sigma$. In particular, the first non-vanishing group is 
$HC^{\cont}_{2n-2-D}(V) \cong H^{D}(M\setminus \Sigma)$.
\end{proof}

\begin{prop}\label{degree}
If $(M^{2n},\omega,J,\Sigma,k)$ is a subcritical polarization such that 
$c_{1}(M\setminus \Sigma) = 0$, then the degree of polarization is $1$.
\end{prop}
\begin{proof}
The proof is already sketched in \cite{EliashbergGiventalHofer}. We fill in some details.
Consider the normal bundle $\mathcal{L}$ over $\Sigma$ and let $V\subset \mathcal{L}$ 
be the unit circle bundle. To keep orientation of $V$ consistent with that of the subcritical complement, 
$\mathcal{L}$ is regarded as a symplectic filling with a \emph{negative} end symplectomorphic to $V\times (-\infty,0]$. 
The Reeb flow on $V$ is fibrewise rotation, and each fibre of $V$ is a Reeb orbit.
We are now in the Morse--Bott setup. We refer to \cite{Bourgeois} and \cite{compactness} for precise
definitions, index computations, and compactness results. 

Choose an auxiliary Morse function $F$ on $\Sigma$. The chain complex for 
$HC^{\cont}(V)$ is generated by Reeb orbits (including their multiple
covers) which lie above the critical points of $F$. The differential counts the number of \emph{generalized}
holomorphic curves of expected dimension $1$. A generalized holomorphic curve consists of multiple stories,
each story is either a holomorphic curve in $V\times \mathbb{R}$ asymptotic to Reeb orbits over points
on $\Sigma$ (not necessarily the critical points of $F$), or fragments of gradient trajectories of $F$ 
connecting points on $\Sigma$. Furthermore the set of Reeb orbits on the positive end and negative end of 
consecutive stories match.

The symplectization of $V$ can be identified with the complement of the zero-section of $\mathcal{L}$.
Assume at first that there is a regular complex structure $J$ on $\Sigma$ for its Gromov-Witten theory 
(for example, if $\Sigma$ is semi-positive). Choose a $\mathbb{C^{*}}$-invariant complex structure 
$\widetilde{J}$ on $\mathcal{L}$ such that the projection 
$\pi\colon (\mathcal{L},\widetilde{J}) \rightarrow (\Sigma,J)$ is holomorphic.
A $\widetilde{J}$-holomorphic curve $u$ in $V\times \mathbb{R}$ projects to a $J$-holomorphic 
sphere $\pi\circ u\colon \mathbb{C}P^1\rightarrow \Sigma$, and can be viewed
as a meromorphic section of the induced complex line bundle 
$(\pi\circ u)^*\mathcal{L}$ over $\mathbb{C}P^1$. The location and multiplicities of the poles
and zeros are determined by the multiplicities of the asymptotic Reeb orbits at positive and negative
infinity.

The linearized $\overline{\partial}$ operator $D_{u}$ splits as the sum of $D_{(\pi\circ u)}$,
the linearized $\overline{\partial}$ operator for the Gromov-Witten theory on $(\Sigma, J)$, and $D_{\mathcal{L}}$,
the linearized $\overline{\partial}$ operator for meromorphic sections of the 
induced line bundle $(\pi\circ u)^*\mathcal{L}$. By regularity assumption of $J$, 
$D_{\pi\circ u}$ is surjective. By Riemann--Roch, $D_{\mathcal{L}}$ is always surjective. Therefore any 
$\widetilde{J}$-holomorphic curve $u$ is regular. Moreover, since meromorphic sections come in 
$\mathbb{C}^{*}$-families, the index of $u$ is at least $2$. Hence the index-$1$ generalized holomorphic curves
consist entirely of gradient trajectories of the auxiliary Morse function $F$. Therefore 
$HC^{\cont}(V)\cong \bigoplus_{i=1}^{\infty}H(\Sigma)$, with a copy of $H(M)$ for each orbit multiplicity.
In particular, the simple orbit $\gamma$ over the minimum of $F$ is a cycle. By Lemma \ref{h1}, $\gamma$
is in the trivial class of $H_1(V)$, and therefore has a capping surface on $V$. 
Together with the disk which $\gamma$ bounds in the fibre of $\mathcal{L}$, we have a closed surface in $M$ 
which intersects $\Sigma$ once. Therefore $\Sigma$ cannot represent a multiple of an 
integral class in $H^{2}(M)$, so the degree of the polarization must be $k=1$.

For the general case we can apply either the domain dependent perturbation of \cite{cieliebakmohnke}, or
the polyfold theory of \cite{GWpolyfold}, to achieve transversality for the Gromov-Witten theory on $\Sigma$, and
hence surjectivity of $D_{\pi\circ u}$. The argument then follows in identical fashion.
\end{proof}

\begin{prop}\label{monotone}
If $(M^{2n},\omega,J,\Sigma,k)$ is a subcritical polarization such that 
$c_{1}(M\setminus \Sigma) = 0$, then $\Sigma$ is monotone.
\end{prop}
\begin{proof}
By Lefschetz theorem and Thom isomorphism, $c_1(TM)$ vanishing on $M\setminus \Sigma$ 
implies $c_1(TM)$ is Poincar\'{e} dual to a multiple of the divisor $\Sigma$. On $\Sigma$,
$c_1(TM) = c_1(T\Sigma) + c_1(\mathcal{L})$.
Being the Euler class, $c_1(\mathcal{L})$ is Poincar\'{e} dual to $\Sigma$. Hence
$c_1(T\Sigma)$ is also a multiple of $[\Sigma]$. We need to check that it is in fact 
a positive multiple.

By Proposition \ref{degree}, the degree of a subcritical polarization is $k=1$, and
\begin{equation}\label{eq:grading}
HC^{\cont}(V)\cong \bigoplus_{i=1}^{\infty}H(\Sigma).
\end{equation}

We compute the grading shift of each copy of $H_{*}(\Sigma)$. Let $\gamma$ be a simple fibre
on $V$. Let $A = A_1\cup A_2$ be the union of the disk $A_1$ which $\gamma$ bounds on the fibre of $\mathcal{L}$, and a capping 
surface $A_2$ of $\gamma$ on $V$. On $\gamma$, the contact distribution $\xi$ is $T\Sigma$. If we trivialize $\xi$ over $A_1$ and
take the split trivialization, then
the linearized Reeb flow is identity, so the Maslov index is $0$. The Maslov index of $\gamma$ 
with respect to a trivialization over $A_2$ is therefore $2c_1(T\Sigma, A)$. By \cite{Bourgeois}, the Reeb orbit $\gamma$
over an index-$i$ critical point of $F$ has Maslov index $2c_1(T\Sigma, A) - \frac{1}{2}\dim(\Sigma) + i$. Therefore 
the grading of $\gamma$ in contact homology is 
$$(n-3)+2c_1(T\Sigma, A) - (n-1) + i = 2c_1(T\Sigma, A) + i -2.$$
Similarly, the grading for a multiplicity-$l$ orbit is 
\begin{equation}\label{eq:grading2}
2lc_1(T\Sigma, A) + i -2.
\end{equation}

By \eqref{eq:grading} and \eqref{eq:grading2}, if $c_1(T\Sigma, A)$ is negative, then $HC^{\cont}(V)$ is non-trivial in
arbitrarily low degrees, contradicting the second part of Lemma \ref{h1}. If  $c_1(T\Sigma, A) = 0$, then the lowest
graded non-vanishing group of $HC^{\cont}(V)$ is infinitely generated, again contradicting Lemma \ref{h1}. Therefore
$c_1(T\Sigma, A)$ is positive, and the lowest graded non-vanishing group is
\begin{equation}\label{eq:compare}
H_{0}(\Sigma)\cong HC^{\cont}_{2c_1(T\Sigma, A)-2} \cong HC^{\cont}_{2n-2-D}\cong H^{D}(M\setminus \Sigma). 
\end{equation}

Note that $A$ intersects $\Sigma$ positively (by the choice of orientation on $A_1$). Therefore $<[\Sigma],A>$
is also positive, and $c_1(T\Sigma)$ is a positive multiple of $\omega=[\Sigma]$.
\end{proof}

\begin{proof}[\emph{\textbf{Proof of Theorem \ref{unirule}}}]
Let $V'\subset M\setminus \Sigma$ be a suitable choice of hypersurface satisfying Proposition \ref{tech}.
Split $M$ along $V$ and $V'$ into a composition of three symplectic cobordisms: a negative filling $\mathcal{L}$
with Morse--Bott boundary asymptotics $V\times (-\infty,0]$, a cobordism 
$\overrightarrow{VV'}$ between $V$ and $V'$, and a positive subcritical Stein filling $M\setminus \Sigma$ 
with boundary asymptotics $V'\times [0,\infty)$ satisfying Proposition \ref{tech} and Theorem \ref{main}. 

In $\mathcal{L}$, consider the moduli space $\mathcal{M}$ of generalized holomorphic planes 
with the negative end asymptotic to the simple Reeb orbit over the minimum of $F$, and in the
homology class of a fibre. The expected dimension of 
$\mathcal{M}$ is $(n-3)-(2c_1(T\Sigma, A) - (n-1))+ 2c_1(TM, A)= 2n-2.$
Since $\Sigma$ is monotone, there exists a regular 
complex structure for its Gromov--Witten theory. Using the same argument as in the proof of 
Proposition \ref{degree}, we see that $\mathcal{M}$ is diffeomorphic to $\Sigma$: 
each fibre $S$ of $\mathcal{L}$, together with the Morse
trajectory from $\pi(S)$ to the minimum of $F$, is a regular generalized holomorphic plane. In particular,
if we add a marked point and require it to pass through a fixed point, then there is a unique rigid
holomorphic plane in $\mathcal{M}$ satisfying this condition. In terms of correlators, this
means
$$\int_{\mathcal{M}_{\gamma}}\ev^{*}(\theta_1) = 1,$$
where $\theta_1$ is Poincar\'{e} dual to a point.

In the cobordism $\overrightarrow{VV'}$ between $V$ and $V'$, the isomorphism between 
$HC^{\cont}_{2c_1(T\Sigma, A)-2}$ and $HC^{\cont}_{2n-2-D}$ is realized on chain level by
a number of rigid cylinders between $\gamma\in V$ and $\gamma'_i\in V'$, such that
$[\sum c_i\gamma_i]$ is a generator of $HC^{\cont}_{2n-2-D} \cong H^{D}(M\setminus \Sigma)$
of on homology (the exact representative $\sum c_i\gamma_i$
will depend on the choice of perturbation).

In the positive filling $M\setminus \Sigma$, let $\theta_2$ 
be Poincar\'{e} dual to a generator of $H_{D}(M\setminus \Sigma)$.
By Theorem \ref{main}, 
$$\sum c_{i}\int_{\mathcal{M}_{\gamma_i}}\ev^*(\theta_2) = 1.$$
In other words, algebraically there is one rigid holomorphic plane asymptotic to one of the 
$\gamma_i$'s, and passing through a cycle representing the generator of $H_{D}(M\setminus \Sigma)$.

These holomorphic curves glue together to give 
(algebraically) one holomorphic $S^2$ in $M$ with $2$ marked points, one marked point 
passes through a fixed point, and the other passes through the generator of $H_{D}(M\setminus \Sigma)$.
Conversely, if $S^2$ is a holomorphic sphere in $M$ in a homology class $A$ whose intersection number
with $\Sigma$ is $1$, with the above marked point requirements, then as we split $M$ along $V$ and $V'$,
$S^2$ must degenerate into the moduli spaces described above. Therefore $M$ has a non-vanishing
Gromov--Witten invariant
$$\sum_{\omega (A)=1}\int_{\mathcal{M}_{g=0,m=2,A}}\ev_{1}^{*}(\theta_{1})\ev_{2}^{*}(\theta_{2})=1.$$
In particular $M$ is uniruled.
\end{proof}

\section{Proof of Proposition \ref{tech}}

Let $M^{2n}$ be a split Weinstein manifold of the form 
$$(M^{2n},\omega,Y,f)=(M'\times \mathbb{C},\  \omega' + dx_n\wedge dy_n, \ Y'+\frac{1}{2}(\del{x_n} + \del{y_n}), \ f' + \kappa(x_n^2+y_n^2))$$ 
where $(x_n,y_n)$ is the Euclidean coordinate on the $\mathbb{C}$ factor. 
We may assume that 
handles are attached in increasing order of Morse indices, the attaching sphere of a $k$-handle only passes
through handles of indices strictly less than $k$, and there is a single $0$-handle.

First we construct a compatible complex structure $J$ for an isolated index-$k$ symplectic handle $U$. Fix a positive integer $m_0$.
Let $$V=\partial U=\left\{\sum_{i=1}^{k}(b_i x_{i}^2 - b'_i y_{i}^2) + \sum_{j=k+1}^{n-1} a_{j}(x_{j}^2+y_{j}^2) + \kappa(x_{n}^2+y_{n}^2)= 1\right\}$$
be the corresponding index-$k$ contact handle with $\frac{\kappa}{m_0} >> b_i, a_j$.  
As in Proposition \ref{prop1}, there is a distinguished simple Reeb orbit $\gamma$ on $V$ such that 
all low index orbits on $V$ are multiple covers of $\gamma$,
$$\gamma = \{x_{n}^2+y_{n}^2=\frac{1}{\kappa},\  x_{1}=y_{1}=\dots =x_{n-1}=y_{n-1}=0\}.$$

Let $B_{\kappa}$ be the %$\frac{1}{\sqrt{\kappa}}$-
neighbourhood of the core isotropic disk $D_{\st}$ of the $k$-handle:
$$B_{\kappa} = \left\{\sum_{i=1}^{n-1} \frac{\kappa}{m_0}x_{i}^2 + \sum_{j=k+1}^{n-1} \frac{\kappa}{m_0} y_{j}^2 
+ \kappa(x_{n}^2+y_{n}^2) \leq 1\right\}.$$ 
$B_{\kappa}$ lies in $U$, and is tangent to $V$ along $\gamma$.

Let $G^{m}_{\kappa}$ be the ellipsoid 
$$G^{m}_{\kappa}=\left\{\sum_{i=1}^{n-1}\frac{\kappa}{m}(x_{i}^2+y_{i}^2) + \kappa(x_{n}^2+y_{n}^2) \leq 1\right\}$$
Essentially $B_{\kappa}$ is the union of $G^{m_0}_{\kappa}$-ellipsoids around each point of the 
core isotropic disk $D_{\st}$. Note that $G^{m}_{\kappa}\subset G^{m_0}_{\kappa}$ if $m<m_0$.

On $\gamma$, the Liouville vector field, Reeb vector field, and contact distribution are:  
$$Y = x_n\del{x_n} + y_n\del{y_n},\ \  R = 2\kappa(x_n\del{y_n}-y_n\del{x_n}),\ \  \xi = T\mathbb{C}^{n-1}.$$ 

In a neighbourhood $\{\sum_{i=1}^{n-1}x_i^2+y_i^2 < \epsilon\}$ of $\gamma$ on $V$, the
contact distribution is a graph over $T\mathbb{C}^{n-1}$.
The projection $\mathbb{C}^{n} = \mathbb{C}^{n-1}\times \mathbb{C}\rightarrow \mathbb{C}^{n-1}$ induces an isomorphism between 
$\xi$ and $T\mathbb{C}^{n-1}$.
Choose $J$ to be an almost complex structure such that
\begin{enumerate}
\item inside $B_{\kappa}$, $J$ is the standard complex structure on $\mathbb{C}^{n}$;
\item on $V$, $J\xi = \xi$, $JY = \frac{1}{2\kappa}R$, and $JR = -2\kappa Y$;
\item in a sufficiently small neighbourhood $\{\sum_{i=1}^{n-1}x_i^2+y_i^2 < \epsilon'\}$ of $\gamma$ on $V$, 
$J|_{\xi}$ is the lift of the standard complex structure on $\mathbb{C}^{n-1}$;
\item between $B_{\kappa}$ and $V$, $J$ smoothly interpolates between the standard complex structure and
the almost complex structure on $V$ through $\omega$-tame almost complex structures;
\item $J$ is invariant under rotation in the $\{x_n,y_n\}$-coordinate plane.
\end{enumerate}

\begin{rmk}\label{localJ}
\emph{
Along $\gamma$, the two ways to define $J$ agree, so $J$ is well defined. 
Furthermore, $J$ tames $\omega$ on $\gamma$, so
$J$ tames $\omega$ in a small neighbourhood $\{\sum_{i=1}^{n-1}x_i^2+y_i^2 < \epsilon'\}$ of $\gamma$. 
$J$ is not quite a compatible complex structure, since $Y$ is paired with a multiple of $R$. 
To turn $J$ into a compatible complex structure, choose a smooth function $h(t)$ such that $h(0)=\frac{1}{2\kappa}$ and $h(1)=1$.
In the cylindrical end $V\times [0,\infty)$, for $t\in [0,1]$, let $J|_{V \times \{t\}}$ be $\mathbb{R}$-invariant on the contact
distribution $\xi$, $J|_{V \times \{t\}}Y = h(t)R$, and $J|_{V \times \{t\}}R = -\frac{1}{h(t)}Y$. For $t>1$, let $J$ be 
$\mathbb{R}$-invariant. This gives a compatible complex structure $J$ inside the symplectic handle. Moreover it is 
not hard to check that the
associated vector field $Z$ for $J|_{V}$ vanishes exactly at the origin and is Morse of index-$k$, as required in Remark \ref{rmk2}.
}
\end{rmk}

\begin{defn}
\emph{
A \emph{model complex structure $J$} is an $S^1$-invariant compatible complex structure constructed above.  
}
\end{defn}

\begin{lem}\label{regularity}
For a model complex structure $J$, there exists a regular central vertical holomorphic plane
$$u(s,t)=(0,0,\dots, 0,x_n(s,t),y_n(s,t)).$$ 
Furthermore the $m$-fold branch covers $u_m(z) = u(z^m)$ are also regular.
\end{lem}
\begin{proof} By the definition of $J$, $J$ is split on the central vertical plane $\{(0,0,\dots, 0,x_n,y_n)\}$. 
It is the standard complex structure on the first $(n-1)$ copies of $\mathbb{C}$, and by uniformization there is a
holomorphic map $u(s,t)=(0,0,\dots, 0,x_n(s,t),y_n(s,t))$.
The linearized $\overline{\partial}$ operator at $u$, $D_{u}$, can be directly
computed in local coordinates: $$D_{u}v = \eta ds - J(u)\eta dt$$
where $v = (\mu_{1},\nu_1,\dots, \mu_{n},\nu_{n})\in u^*(TM)$ and
$$\eta = \frac{1}{2}(\del{s}v+J(u)\del{t}v+(\del{v}J)(u)\del{t}(u))$$
We need to check the surjectivity of the differential operator 
$\eta\colon \mathbb{R}^{2n}\rightarrow \mathbb{R}^{2n}$.

Fix a point $p=(0,0,,\dots, 0,P,Q)$. First assume $p$ lies
on $V = V\times \{0\}$, take a a path on $V$
$$\alpha(\tau)=(\tau,0,\dots,0,x_{n}(\tau),y_{n}(\tau)),$$
where $$x_{n}(\tau)=P\sqrt{1-b_1\tau^2},\ \  y_{n}(\tau)=Q\sqrt{1-b_1\tau^2}.$$

The Liouville and Reeb fields are given by:
$$Y(\tau)= \left\{
\begin{array}
[c]{ll}
2\tau\del{x_1}+\frac{1}{2}x_{n}(\tau)\del{x_n}+\frac{1}{2}y_{n}(\tau)\del{y_n}, & \text{if Morse index is zero} \\
\frac{1}{2}\tau\del{x_1}+\frac{1}{2}x_{n}(\tau)\del{x_n}+\frac{1}{2}y_{n}(\tau)\del{y_n}, & \text{if Morse index is greater than zero}\\
\end{array} \right.
$$

$$R(\tau)= \left\{
\begin{array}
[c]{ll}
\frac{1}{3b_1\tau^2+1}(2b_{1}\tau\del{y_{1}}- 2\kappa y_n(\tau)\del{x_n}+
2\kappa x_{n}(\tau)\del{y_n}),& \text{if Morse index is zero} \\
2a_{1}\tau\del{y_1}- 2\kappa y_n(\tau)\del{x_n}+
2\kappa x_{n}(\tau)\del{y_n}, & \text{if Morse index is greater than zero}\\
\end{array} \right.
$$

In any case, $Y$ and $R$ only involve $\{\partial x_{1}, \partial y_1, \partial x_n, \partial y_n\}$. Furthermore,
the contact distribution contains $\{\del{x_{2}},\del{y_2},\dots,\del{x_{n-1}},\del{y_{n-1}}\}$. 
By its definition, $J$ is the standard complex structure on 
$\{\del{x_{2}},\del{y_2},\dots,\del{x_{n-1}},\del{y_{n-1}}\}$.

Hence $\del{(1,0,\dots,0)}J$ is non-zero only for the submatrix spanned by 
$\{\partial x_{1}, \partial y_1, \partial x_n, \partial y_n\}$:
$$\del{(1,0,\dots,0)}J =
\left[
\begin{array}{ccccccc}
* & * & 0  & \dots & 0 & * & *\\
* & * & 0  & \dots & 0 & * & *\\
0 & 0 & 0  & \dots & 0 & 0 & 0\\
\vdots& &  & \ddots & & & \vdots \\
0 & 0 & 0  & \dots & 0 & 0 & 0\\
* & * & 0  & \dots & 0 & * & *\\
* & * & 0  & \dots & 0 & * & *
\end{array}
\right]
$$

where possible non-zero terms are marked by $*$.
 
Also $\del{t}(u) = (0,\dots, 0, \del{t}x_{n},\del{t}y_{n})$, hence

$$(\del{(1,0,\dots,0)}J)\del{t}(u)|_{(0,0,\dots, P,Q)}=
\left[ \begin{array}{c} 
A_1 \\ A_{2}\\ 0 \\ \vdots \\ 0 \\ C_1 \\ C'_1\end{array} \right]$$

Derivatives in the other directions are computed similarly.

If $p$ lies on other levels $V\times \{x\}$, the computation is identical. 
If $p$ lies inside $B_{\kappa}$, then $D_{u}$ is the standard 
Cauchy--Riemann operator. Therefore $D_{u}$ overall takes the form

$$
D_{u}\left[
\begin{array}{c}
\mu_1 \\ \nu_1 \\  \vdots \\ \mu_{n-1} \\ \nu_{n-1} \\ \mu_n \\ \nu_n
\end{array}\right]
= \left[
\begin{array}{c}
\del{s} \mu_1 - \del{t} \nu_1 + A_1(s,t)\mu_1+B_1(s,t)\nu_1\\
\del{s} \nu_1 + \del{t} \mu_1 + A_2(s,t)\mu_1+B_2(s,t)\nu_1\\
\vdots \\
\del{s} \mu_{n-1} - \del{t} \nu_{n-1} + A_{2n-3}(s,t)\mu_{n-1}+B_{2n-3}(s,t)\nu_{n-1}\\
\del{s} \nu_{n-1} + \del{t} \mu_{n-1} + A_{2n-2}(s,t)\mu_{n-1}+B_{2n-2}(s,t)\nu_{n-1}\\
\del{s} \mu_{n} - \del{t} \nu_{n} + A_{2n-1}(s,t)\mu_{n}+B_{2n-1}(s,t)\nu_{n}+F_1(\mu_1, \nu_1,\dots, \mu_{n-1},\nu_{n-1},s,t)\\
\del{s} \nu_{n} + \del{t} \mu_{n} + A_{2n}(s,t)\mu_{n}+B_{2n}(s,t)\nu_{n}+F_2(\mu_1, \nu_1,\dots, \mu_{n-1},\nu_{n-1},s,t)\\
\end{array}\right]
$$

The operator $D_{u}$ is upper triangular, in fact except the for last row, $D_u$ splits into a direct sum of real Cauchy--Riemann operators
on $\mathbb{C}$. Therefore $D_{u}$ is surjective. The computation for the multiple cover $D_{u_m}$ is identical.
\end{proof}

We need a generalization of the Monotonicity Lemma for the standard complex structure in $\mathbb{C}^n$. We 
thank S. Lisi for the following simple proof.

\begin{lem}\label{mono}
Let $G\subset \mathbb{C}^{n}$ be the ellipsoid $\left\{\sum_{j=1}^{n-1}|z_{j}|^2 + m|z_{n}|^2 \leq r^2\right\}$, $D$
the unit disk in $\mathbb{C}$ and $\omega$ the standard symplectic form. 
Suppose $u=(u_1,\dots, u_n)\colon D\rightarrow G$ is a holomorphic disk such that 
$u(\partial D)\subset \partial G$, $u(0) = 0$, and $\frac{\partial^{l} u_n}{\partial z_n^l}(0) =0$ for $l=1,\dots m-1$.
Then $$\int_{D}u^{*}\omega \geq \pi r^2.$$
\end{lem}
\begin{proof}
Without loss of generality we may assume $r=1$. Expand each component of $u$ into power series 
$u_j=\sum_{k=1}^{\infty}A_{kj}z^k$. Then 
\begin{eqnarray*}
\int_{D}u^{*}\omega &=& \frac{i}{2}\sum_{j=1}^{n}\int_{D}u_{j}^{*}(dz_j\wedge d\bar{z}_j)\\
&=& \frac{i}{2}\sum_{j=1}^{n}\int_{\partial D}u_{j}^{*}(z_jd\bar{z}_j)\\
&=& \frac{i}{2}\sum_{j=1}^{n}\int_{\partial D}\left(\sum_{k=1}^{\infty}A_{kj}z^k\right)
\left(\sum_{k=1}^{\infty}k\overline{A}_{kj}\bar{z}^{k-1}d\bar{z}\right) 
\end{eqnarray*}
by the residue theorem, the only non-vanishing integral is 
\begin{equation*}
\int_{\partial D}z d\bar{z} = -2\pi i
\end{equation*}
Therefore
\begin{equation*}
\int_{D}u^{*}\omega = \pi \sum_{j=1}^{n}\sum_{k=1}^{\infty}k|A_{kj}|^2
\end{equation*}

On $\partial D$,
\begin{eqnarray*}
1 &=& \sum_{j=1}^{n-1}|z_{j}|^2 + m|z_{n}|^2\\
&=& \sum_{j=1}^{n-1}z_{j}\bar{z}_{j} + m z_{n}\bar{z}_{n}\\
&=& \sum_{j=1}^{n-1}\left(\sum_{k=1}^{\infty}A_{kj}z^k\right)\left(\sum_{k=1}^{\infty}\overline{A}_{kj}\bar{z}^k\right)
+m \left(\sum_{k=1}^{\infty}A_{kn}z^k\right)\left(\sum_{k=1}^{\infty}\overline{A}_{kn}\bar{z}^k\right)\\
&=& \sum_{j=1}^{n-1}\left(\sum_{k=1}^{\infty}A_{kj}z^k\right)\left(\sum_{k=1}^{\infty}\overline{A}_{kj}{z}^{-k}\right)
+m \left(\sum_{k=1}^{\infty}A_{kn}z^k\right)\left(\sum_{k=1}^{\infty}\overline{A}_{kn}{z}^{-k}\right)
\end{eqnarray*}
Equating Fourier coefficients, we have
$$1 = \sum_{j=1}^{n-1}\sum_{k=1}^{\infty}|A_{kj}|^2 + m\sum_{k=1}^{\infty}|A_{kn}|^2$$
Since $A_{kn}=0$ for $k=1,\dots, m-1$,
\begin{equation*}
\int_{D}u^{*}\omega = \pi \sum_{j=1}^{n}\sum_{k=1}^{\infty}k|A_{kj}|^2 \geq 
\pi \left(\sum_{j=1}^{n-1}\sum_{k=1}^{\infty}|A_{kj}|^2 + m\sum_{k=1}^{\infty}|A_{kn}|^2 \right)
=\pi.
\end{equation*}
Moreover, equality holds iff $u(z)$ is of the form $(C_1z,\dots,C_{n-1}z,C_nz^m)$. 
\end{proof}

Let the shrunken core $D_{\st}^{\kappa}$ of the handle be
$$D_{\st}^{\kappa}=\left\{\sum_{i=1}^{k}y_i^2 \leq 1-\frac{2m_0}{\kappa}, x_i=x_j=y_j=0, k+1\leq j\leq n\right\}.$$

\begin{lem}\label{insidehandle}
For a model complex structure $J$ and each $m\leq m_0$, the $m$-fold branch cover of the central vertical plane, 
$u_m$, is the only holomorphic plane asymptotic to the Reeb orbit $\gamma^m$, and passing through some point 
on  $D_{\st}^{\kappa}$ with ramification index $m$ with respect to the vertical complex direction $\C$.
\end{lem}
\begin{proof}
This argument is essentially the definition of $E_{\lambda}$-energy of $u$ in \cite{compactness} and Lemma \ref{mono}.
Let $u$ be a holomorphic plane asymptotic to $\gamma$ and passing through some point $p$ on $D_{\st}^{\kappa}$. By construction of $J$,
around each point $p$ on $D_{\st}^{\kappa}$, there is an ellipsoid $G^{m}_{\kappa}(p)$ on which
$J$ is standard. By Lemma \ref{mono}, 
$$\int_{u \cap G^{m}_{\kappa}(p)}\omega \geq \frac{m\pi}{\kappa}.$$
If $u$ is not a multiple cover of the central vertical plane, then some part of $u$ lies between $V$ and $G^{m}_{\kappa}(p)$, so
$$\int_{u \cap U}\omega > \frac{m\pi}{\kappa}.$$
Let $f_{\delta}(t)$ be a smooth function such that 
\begin{enumerate}
\item $f'_\delta(t)>0$,
\item $f_\delta(t) = e^t$ near $0$,
\item $f_\delta(t)\rightarrow 1+ \delta$ as $t \rightarrow \infty$.
\end{enumerate}

Let $\alpha_\delta$ be the 1-form such that $\alpha_\delta = \alpha$ on $U$, and $\alpha_\delta = f_\delta(t)\cdot\alpha|_{V}$ on
the cylindrical end $V\times [0,\infty)$. It is easy to see that $J$ tames $d\alpha_\delta$. 

By Stoke's Theorem 
$$\int_{u}d\alpha_{\delta}= (1+\delta)\int_{\gamma^m}\alpha|_{V} = (1+\delta)\frac{m\pi}{\kappa}.$$
Since $d\alpha_{\delta} = d\alpha = \omega$ on $U$,
$$\int_{u \cap U}\omega = \int_{u \cap U}d\alpha_{\delta} < \int_{u}d\alpha_{\delta} = (1+\delta)\frac{m\pi}{\kappa}.$$
Let $\delta \rightarrow 0$, we have a contradiction unless $u$ is the $m$-fold branch cover of the central vertical plane.
\end{proof}

\begin{rmk}
\emph{
The almost complex structure $J$ is only Lipschitz continuous. The calculation of the derivatives 
of $J$ in Lemma \ref{regularity} shows that they are discontinuous on $\gamma$: 
$\del{(1,0,\dots,0)}J$ vanishes if computed along a path 
on $\partial B_{\kappa}$, and is non-zero if computed along a path on $V$. Nevertheless $\gamma$ is the 
only place of non-smoothness. The operator $D_{u}$ is still well defined as
an operator between Sobolev spaces, and is surjective.
If one prefers to work with smooth complex structures, then choose a sequence of of smooth complex structures 
$J_{\delta} \rightarrow J$ such that each $J_{\delta}$ is a smoothing of $J$ and coincides with $J$ 
outside a $\delta$ neighbourhood of $\gamma$. The discontinuities of $J$ are in the normal directions
of the central vertical plane. Therefore $J_{\delta}$ can be chosen to coincide with $J$ on $u_0$, so $u_0$
is holomorphic for each $J_{\delta}$. Let $D_{u}^{\delta}$ be the linearized $\overline{\partial}$ operator at $u_0$
for the complex structure $J_{\delta}$. Then $D_{u}^{\delta} \rightarrow D_{u}$, thus for all sufficiently small 
$\delta$, $D_{u}^{\delta}$ is surjective, and there is a uniform bound on the right inverses of $D_{u}^{\delta}$.
It follows that $u_{0}$ is a regular holomorphic plane for all $\delta$ sufficiently small. Moreover, 
there is a uniform neighbourhood of $u_{0}$, independent of $\delta$, where there is no other holomorphic 
plane. Compactness of holomorphic curves applies to the sequence
$J_{\delta} \rightarrow J$, hence if there is no $J$-holomorphic curve in a certain region, then there is no 
$J_{\delta}$-holomorphic curves in that region for all small $\delta$. For simplicity we will work with the non-smooth $J$,
knowing that there is always a nearby smooth $J_{\delta}$ with the same properties. 
}					
\end{rmk}

We now prove Proposition \ref{tech}.

\begin{proof}[\emph{\textbf{Proof of Proposition \ref{tech}}}]
Fix $m_0$ so that for Reeb orbits of index at most $N$ are of the form $\gamma_{p}^{m}, m<m_0$, as in Proposition 
\ref{prop1}. 
We will inductively reconstruct $(M',f')$ by symplectic handle attachments.
Simultaneously we will construct a compatible complex structure $J$ on $(M, f) = (M'\times \mathbb{C}, f'+\kappa (x_{n}^2+y_{n}^2))$. 
The idea is that the new handles to be attached to $M'$ have shape parameters and Morse function extensions depending on $J$, 
so that a holomorphic plane asymptotic to $\gamma_{p}^{m}$ does not have enough energy to pass through the unstable manifolds of 
critical points of lower indices. 

Let $M'_k\subset M'$ consists of handles of index at most $k$, $f'_{k}$ be a Morse function on $M'_k$ quadratic
on handles, $J_k$ a compatible complex structure on $(M_k,f_k) =(M'_k\times \mathbb{C}, f'_{k}+\kappa(x_n^2+y_n^2))$, 
$W_k\subset M_k$ the union of the unstable submanifolds of index up to $k$, and $L_{k}$ a neighbourhood of $W_k$ such that
$J_k(p,z) = J'_k(p)\oplus i$ is split on $L_k$. We will inductively construct $(M'_k, f'_k, L_k, J_k)$.

For $k=0$, $M'_0$ is the standard index-$0$ handle together with a function $f'_0 = \sum_{j=1}^{n-1}a_j(x_{j}^2+y_{j}^2)$.
Fix $\kappa >> a_j$, and let $J_0$ be a model compatible complex structure on $M_0$. $W_0$ is the origin, 
and $L_0$ is defined to be the $G^{m_0}_{\kappa}$-ellipsoid around the origin where $J_0$ is standard. 

Attach the $(k+1)$-handles to $M'_k$ by symplectic handle attachment. We may assume that the attaching
isotropic spheres are pairwise disjoint, and for dimension reason, also disjoint from the set of Reeb orbits 
$\{\gamma_{p}\colon \indexof (p) \leq k\}$. 
Let $$U=\left\{\sum_{i=1}^{k+1}(b_i x_{i}^2 - b'_i y_{i}^2) + \sum_{j=k+2}^{n-1} a_{j}(x_{j}^2+y_{j}^2) \leq 1\right\}$$ 
be such a $(k+1)$-handle. If we extend $f'_k$ quadratically to $U$ as in Remark \ref{extendf}, 
we attach a $(k+1)$-handle (lying inside $U\times \mathbb{C}$) to $M_k$ as well. Moreover the shape of the handle is 
controlled by the choice of the Morse function extension, in particular it could be thinned in the $\mathbb{C}$-factor 
by simply increasing the critical value in $U$ to approach $1$. 
In fact any sufficiently thin shape can be achieved. Suppose $\widetilde{V}$ is a handle of the shape
$$\widetilde{V}=\left\{\sum_{i=1}^{k+1}(B_i x_{i}^2 - B'_i y_{i}^2) + \sum_{j=k+2}^{n-1} A_{j}(x_{j}^2+y_{j}^2) + C(x_n^2+y_n^2) \leq 1\right\}$$
with $C >> B_i, A_j >> B'_i$ such that 
$$V =\left\{\sum_{i=1}^{k+1}(B_i x_{i}^2 - B'_i y_{i}^2) + \sum_{j=k+2}^{n-1} A_{j}(x_{j}^2+y_{j}^2) \leq 1\right\}$$ 
lies within $U$. Then the Morse function extension on $U$,
\begin{equation}\label{eq:function}
f_{U}=\frac{\kappa}{C}\left(\sum_{i=1}^{k+1}(B_i x_{i}^2 - B'_i y_{i}^2) + \sum_{j=k+2}^{n-1} A_{j}(x_{j}^2+y_{j}^2)\right)+\left(1-\frac{\kappa}{C}\right)
\end{equation}
gives the handle $\widetilde{V}$.

The diffeomorphism attaching the handle $\widetilde{V}$ to $M_k$ is just $\widetilde{\Phi}=(\Phi, {\rm Id})$, the direct sum
of the attaching map $\Phi$ for $U$ and the identity map on the $\mathbb{C}$-factor. 
It remains unchanged as we vary the shape of $\widetilde{V}$. Modify $J_k$ in a small neighborhood of the unstable disks 
flowing from the isotropic spheres on $\partial M_k$ such that $J_k$ becomes split, i.e., 
$J_k(p,z) = J_k(p)\oplus i$ for $p\in M'$ and $|z|<\epsilon$. Since $J_k$ is already assume to be split on $L_k$, this can
be done away from $L_k$. 

Choose one such thin $\widetilde{V_i}$ for each $(k+1)$-handle, define $J$ on $M_k\cup \{\widetilde{V_i}\}$ as follows:
\begin{enumerate}
\item
on $M_k$, $J=J_{k}$;
\item
on each $\widetilde{V_i}$, $J$ equals a model complex structure in the region $\{\sum_{i=1}^{k+1}y_i^2 \leq \frac{9}{10}\}$;
\item
on each $\widetilde{V_i}$, $J$ interpolates between the complex structure $(\widetilde{\Phi}_i)_{*} (J_k)$ and the model complex structure
in the region $\{\frac{9}{10} < \sum_{i=1}^{k+1}y_i^2 \leq 1\}$.
\end{enumerate}

Since $J_k$ is modified to be split near the isotropic spheres, and the attaching maps $\widetilde{\Phi}_i$ 
is split, $(\widetilde{\Phi}_i)_{*} (J_k)$ is split. It follows that the interpolating
complex structure can also be chosen to be split in some neighbourhood of the core of the handle.

Pick a Riemmanian metric $g'$ on $M'_k\cup \{U_i\}$ which is Euclidean on each $U_i$, and let $g$ be the product of $g'$ with the Euclidean
metric on $\mathbb{C}$. Choose $r_0$ sufficiently small such that $L_k$ contains an injective $r_0$-ball around each point of $W_k$. 
By Monotonicity, there is a constant $C_0$ such that for all $J_k$-holomorphic curve $u$ passing through $p\in W_k$, 
\begin{equation}\label{eq:est1} 
\int_{u\cap B_{r_0}(p)}\omega > C_{0}r_{0}^2
\end{equation}

On $M_k\cup \{\widetilde{V_i}\}$, choose $r_1$ sufficiently small such that there is an injective
$r_1$-ball around each point of $W_{k+1}$, and $J$ is split inside every such $r_1$-ball. Again by Monotonicity,
there is a constant $C_1$ such that for all $r<r_1$ and all $J$-holomorphic curves $u$ passing through $p\in W_{k+1}$, 
\begin{equation}\label{eq:est2}
\int_{u\cap B_{r}(p)}\omega > C_1r^2
\end{equation}

Now we thin all the $\widetilde{V_i}$'s as follows:
\begin{enumerate}
\item
$B'_i=\frac{100m_0}{C_1}, 1\leq i\leq k+1$;
\item
$B_i, A_j, C >> \frac{100m_0}{C_1}$, so that the handle lies inside each $\widetilde{V_{i}}$;
\item
$C >> m_0B_i, m_0A_j, \frac{m_0\pi}{C_0r^{2}_{0}}, \frac{m_0\pi}{C_1r_{1}^2}$, and 
sufficiently large such that there is an injective $\sqrt{\frac{m_0\pi}{C_{1}C}}$-ball
around each $p\in W_{k+1}$ outside the region $\{\sum_{i=1}^{k+1}y_i^2 < \frac{8}{10}\}$ in all $(k+1)$-handles.
\end{enumerate}

Note that at the origin the handle has thickness $\frac{1}{\sqrt{C}}$. The choice of $B'_i$  
ensures the ratio between thickness of the handle at $\{\sum_{i=1}^{k+1}y_i^2 = 8/10\}$ and the 
thickness at origin is sufficiently large, so that such a choice of $C$ exists.

Use this choice of $\{B_i, B'_i, A_j\}$ for all $(k+1)$-handles of $M'_k$, and extend the Morse function
$f'$ to $M'_{k+1}=M'_{k}\cup \{V_{i}\}$ as in \eqref{eq:function}. Let $L_{k+1}\subset M_{k+1}$ be the union of $L_{k}$,
the $\sqrt{\frac{m_0\pi}{C_{1}C}}$-neighbourhood around each $p\in W_{k+1}$ outside the region 
$\{\sum_{i=1}^{k+1}y_i^2 < \frac{8}{10}\}$ of all $(k+1)$-handles, and the $G_{C}^{m_0}$-ellipsoid neighbourhood
around the core isotropic disks $D_{\st}$ of all $(k+1)$-handles. Let $J_{k+1}=J$ on $L_{k+1}$, then
extend $J_{k+1}$ to a compatible complex structure. Note that $J_{k+1}$ is standard in the the 
$G_{C}^{m_0}$-ellipsoid neighbourhood around $\{\sum_{i=1}^{k+1}y_i^2 < 9/10\}$ in each handle, and we may take an
extension of $J_{k+1}$ which is a model complex structure in the region $\{\sum_{i=1}^{k+1}y_i^2 < 9/10\}$.
This completes the induction step.

Suppose $u$ is a holomorphic plane asymptotic to $\gamma_{p}^{m}$, where $p$ is an index-$(k+1)$ critical point,
$m< m_0$, and $u$ passes through a point on $W_k$. By construction, 
$\gamma_{p}^{m}$ has action $\frac{m\pi}{C}<C_0r_{0}^{2}$. However by \eqref{eq:est1},
\begin{equation}\label{est3}
\int_{u\cap L_{k+1}}\omega > C_{0}r_{0}^{2}. 
\end{equation}
Using the same argument as in Lemma \ref{insidehandle}, this is a contradiction.

Since $J_{k+1}$ is a model complex structure inside each $(k+1)$-handle, by Lemma \ref{regularity}, the $m$-fold
cover of the vertical plane over an index-$(k+1)$ critical point $p$ is a regular holomorphic plane asymptotic to 
$\gamma_{p}^m$, which passes through $p\in W_{k+1}$ with ramification index $m$. Let $u$ be a holomorphic 
plane asymptotic to $\gamma_{p}^m$
and passes through $W_{k+1}$. If $u$ passes through a point in $W_{k+1}$ outside the region 
$\{\sum_{i=1}^{k+1}y_i^2 < \frac{8}{10}\}$ in all $(k+1)$-handles, then \eqref{eq:est2} implies
\begin{equation}\label{est4}
\int_{u\cap L_{k+1}}\omega > C_1\left(\sqrt{\frac{m_0\pi}{C_{1}C}}\right)^2 > \frac{m\pi}{C}.
\end{equation}
Also a contradiction. If $u$ passes through $W_{k+1}$ in the region $\{\sum_{i=1}^{k+1}y_i^2 < \frac{9}{10}\}$,
then Lemma \ref{insidehandle} shows that $u$ can only be the $m$-fold branch cover of the vertical plane over $p$.
 
Note that for the energy lower bounds, we only use the complex structure inside $L_{k+1}$, there is 
no restriction on $J_{k+1}$ outside of $L_{k+1}$. Furthermore $L_{k+1}\cap \partial M_{k+1}$ consists of all
the Reeb orbits $\gamma_{p}$ over critical points of index $\leq k+1$. $L_{k+1}$ otherwise lies
in the interior of $M_{k+1}$.

At the end of the induction we have $(M',f',L,J)$ and $(M, f)=(M', f'+\kappa(x_n^2+y_n^2))$, such that
$(M,\partial M, J)$ satisfies parts (4) and (5) of Proposition \ref{tech}. However to satisfy Proposition \ref{prop2}, 
we need to increase $\kappa$ by some large factor $K$ (which depends on $J|_{\partial M}$), as described 
in Remark \ref{rmk3}. Increasing $\kappa$ has the effect of shrinking $M$ in the $\mathbb{C}$-factor.
Let $(M_{K},f_{K})=(M', f'+K\kappa(x_n^2+y_n^2))$, $\Psi_{K}$ be the dilation by 
$\frac{1}{\sqrt{K}}$ in the $\mathbb{C}$-factor, and $J_{K}$ defined on the contact distribution on $\partial M_{K}$ by
pushing forward $J_{K}$ on $\partial M$. Let $L_{K}=\Psi_{K}(L)\subset M_K\subset M$. To define $J_K$ in the filling,
let $J_K = J$ on $L_K$ and then interpolate between $L_K$ and $\partial M_K$. 

Each orbit $\gamma_{p}^m$ on $M_K$ has action equal to $\frac{1}{K}$-th of its action on $M$. Whenever there is an injective
$G_C^{m_0}$-ellipsoid in $L$, there is an injective $G_{KC}^{m_0}$-ellipsoid in $L_{K}$. Since $J_K=J$ on $L_K$, the
Monotonicity constants are unchanged. Therefore the inequalities \eqref{est3} and \eqref{est4} still holds with $L_K$
in place of $L$, and a factor of $\frac{1}{K}$ on both sides. It follows that $(M_K,J_K)$ still satisfies (4) and (5) of 
Proposition \ref{tech}. 

Finally by construction $L$ has the split complex structure as required by part (3) of Proposition \ref{tech}, and it is clear
that one can choose a compatible complex structure which makes $M'$ a complex submanifold (since $M'\cap L$ is already a
complex submanifold, and there is no restriction on $J$ outside $L$). This completes the proof.
\end{proof}

\section{Bourgeois--Oancea Exact Sequence} 
In \cite{Oancea}, Bourgeois and Oancea proved a long exact sequence between linearized contact homology and
symplectic homology, and deduced Theorem \ref{yau1}. 

\begin{thm}[\cite{Oancea}]
There is a long exact sequence
$$\cdots \longrightarrow SH^{+}_{*-(n-3)}(M)\longrightarrow HC_{*}(\partial M) \stackrel{D}{\longrightarrow}
HC_{*-2}(\partial M) \longrightarrow SH^{+}_{*-1-(n-3)}(M)\longrightarrow \cdots$$
\end{thm}

It was not immediately clear how the explicit construction in \cite{Yau} is related to the more 
abstract approach of Bourgeois and Oancea. In this section, for the special setup of contact form and compatible
complex structure given by Proposition \ref{tech}, we determine the degree-$2$ map $D$ in the Bourgeois--Oancea exact sequence.

Following \cite{Oancea}, the differential $D$ can described exclusively in terms of holomorphic 
curves in the symplectization of $\partial M$. Suppose we are in a setup given by Proposition \ref{tech},
so all Reeb orbits are of the form $\gamma_p^m$. Since each orbit is $S^1$-invariant under the rotation
in the vertical $\C$ factor, its geometric image can be canonically identified with the unit circle $S$ on
$\C$. For each orbit $\gamma$ choose a point $e^{i\theta_{\gamma}}\in S, \theta_{\gamma}\in(0, 2\pi)$,
such that if $\gamma$ is an orbit above a critical point $p$, $\gamma'$ an orbit above $p'$, and 
$\indexof(p)>\indexof(p')$, then $\theta_{\gamma} > \theta_{\gamma'}$ (regardless of their multiplicities).

Choose a global polar coordinate $\C\setminus \{0\}$ for a cylinder. Let $L$ be the positive real axis.
Given a parametrized holomorphic cylinder between $\gamma$ and $\gamma'$,
$u = (u', a)\colon \C\setminus \{0\}\rightarrow \partial M\times \R$, define
$\ev^+(u) = \lim_{z\rightarrow \infty, z\in L}u'(z) \in \gamma \cong S$ and 
$\ev^-(u) = \lim_{z\rightarrow 0, z\in L}u'(z) \in \gamma' \cong S$.

The map $D$ is induced by the chain level map
\begin{equation}\label{eq:def1}
\Delta(\gamma)=\frac{1}{\kappa_{\gamma'}}\sum_{\overline{\mu}(\gamma')=\overline{\mu}(\gamma)-2} c_{\gamma,\gamma'}\gamma', 
\end{equation}
where $c_{\gamma,\gamma'}$ is the sum of counts of two types of moduli spaces:
\begin{enumerate}
\item the moduli space $\M_{1}$ of parametrized holomorphic cylinders $u$ 
asymptotic to $\gamma$ and $\gamma'$ such that $\ev^+(u)=e^{i\theta_{\gamma}}$, and 
$\ev^-(u)=e^{i\theta_{\gamma'}}$.
\item the moduli space $\M_2$ of parametrized broken holomorphic cylinders $(u_1,u_2)$
such that $\ev^+(u_1) = e^{i\theta_{\gamma}}$, $\ev^-(u_2)=e^{i\theta_{\gamma'}}$. Furthermore, on the 
intermediate breaking Reeb orbit $\beta$ with $\overline{\mu}(\beta)=\overline{\mu}(\gamma)-1$,
$\{\ev^-(u_1), e^{i\theta_{\beta}}, \ev^+(u_2)\}$ lie in clockwise order.
\end{enumerate}

\begin{lem}\label{lem01}
In a setup given by Proposition \ref{tech}, the contribution from $\M_2$ is zero.
\end{lem}
\begin{proof}
By Proposition \ref{prop2} we have a complete understanding of the moduli spaces of holomorphic
cylinders between Reeb orbits of index difference $1$. They are $S^1$-invariant, run 
between $\gamma_p^m$ and $\gamma_p'^m$ with $\indexof(p) = \indexof(p') - 1$, and
are lifts of trajectories of a gradient-like vector field $Z$ between $p$ and $p'$.
For any
$S^1$-invariant cylinder $u$, $\ev^+(u)=\ev^-(u)$. Hence $\gamma, \beta, \gamma'$ are
Reeb orbits of the same multiplicity over critical points of increasing Morse index.
By the choice of $\theta_{\gamma}$,
$\{\ev^-(u_1), e^{i\theta_{\beta}}, \ev^+(u_2)\} = 
\{e^{i\theta_{\gamma}}, e^{i\theta_{\beta}}, e^{i\theta_{\gamma'}}\}$ indeed lie in clockwise
order. Hence every broken cylinder in $\partial^2(\gamma)$ in the cylindrical contact
homology differential counts. But in our setup the cylindrical contact
homology differential coincides with the Morse cohomology differential, so $\partial^2(\gamma)$
vanishes.
\end{proof}

\begin{rmk}
\emph{
Requiring $\ev^+(u)=e^{i\theta_{\gamma}}$ picks out a parametrized representative for 
each unparametrized holomorphic cylinder in $\M_{\gamma,\gamma'}$ (strictly speaking
there are $\kappa_{\gamma}$ representatives, but as in Remark \ref{carpet} we will
ignore the combinatorial factors). Then $\ev^{-}$ can be viewed as a map from
$\M_{\gamma,\gamma'}$ to $\gamma' \cong S$. Equivalently requiring 
$\ev^-(u)=e^{i\theta_{\gamma'}}$ turns $\ev^{+}$ into a map 
$\M_{\gamma,\gamma'} \rightarrow \gamma \cong S$.
}
\end{rmk}

If $\overline{\mu}(\gamma')=\overline{\mu}(\gamma)-2$, then $\M_{\gamma,\gamma'}$ is
1-dimensional, and $\M_1$ consists of solutions of 
$\ev^{-}(u) = e^{i\theta_{\gamma'}}$ in $\M_{\gamma,\gamma'}$. $\M_{\gamma,\gamma'}$
has of two types of connected components. Let 
$\M_{\gamma,\gamma'}^{S^1}$ be the union of the connected components of 
$\M_{\gamma,\gamma'}$ homeomorphic to a circle, 
and $\M_{\gamma,\gamma'}^{[0,1]}$ the union of those homeomorphic to an 
interval.

\begin{lem}\label{lem02}
In a setup given by Proposition \ref{tech}, $\M_{\gamma,\gamma'}^{[0,1]}$ does
not contain any solution of $\ev^{-}(u) = e^{i\theta_{\gamma'}}$.
\end{lem}
\begin{proof}
Again by Proposition \ref{prop2} we completely understand 
$\M_{\gamma,\gamma'}^{[0,1]}$. The endpoint of the intervals are rigid 
broken cylinders, which are lifts of rigid broken trajectories of $Z$. By standard
Morse theory there are several $1$-parameter families of trajectories of $Z$ which connect pairs
of rigid broken trajectories. They lift to $1$-parameter families of $S^1$-invariant
cylinders which connects the rigid broken cylinders in pairs. Since $S^1$-invariant 
cylinders are regular, these all of $\M_{\gamma,\gamma'}^{[0,1]}$. Hence 
$\M_{\gamma,\gamma'}^{[0,1]}$ consist entirely of $S^1$-invariant cylinders, so
$\ev^{-}(u) = \ev^{+}(u) = e^{i\theta_{\gamma}} \neq e^{i\theta_{\gamma'}}$.
\end{proof}

Therefore by Lemmas \ref{lem01} and \ref{lem02}, the coefficient $c_{\gamma,\gamma'}$
in \eqref{eq:def1} is the same as the winding number of 
$\ev^{-}\colon \M_{\gamma,\gamma'}^{S^1} \rightarrow S$.

\begin{prop}\label{D}
If $M$ is a subcritical Stein manifold with vanishing $c_1$, $a\in HC(\partial M)$, 
$\theta$ a compactly supported closed form on $M$, and $D$ the degree-$2$ differential
in the Bourgeois--Oancea exact sequence, then
\begin{equation}\label{eq:D}
l!\int_{\mathcal{M}_{a}}\ev^*(\theta)\wedge \psi^{l} = (l-1)!\int_{\mathcal{M}_{D(a)}}\ev^*(\theta)\wedge \psi^{l-1}
\end{equation}
\end{prop}
\begin{proof}
Take a set up given by Proposition \ref{tech}. Represent the Poincar\'{e} dual of $\theta$ by a
suitable cycle $\alpha$ and interpret the descendants as counts of curves with ramification indices
condition at the constrained marked point. We will prove \eqref{eq:D} for $l=1$, the proof is identical
for higher values of $l$.

Let $\sum_{i=1}^{k} c_i\gamma_i$ be a representative of $a$. 
We are interested in the holomorphic planes asymptotic to $\gamma_i$, 
and passing through $\alpha$ with ramification index $2$. $\M_{\gamma_i}(\alpha)$, the moduli space of
holomorphic planes asymptotic to $\gamma_i$ 
and passing through $\alpha$, is then $2$-dimensional.

First we trivialize the tautological bundle $L$ over $\M_{\gamma_i}(\alpha)$.
As explained in Section $2$, this amounts to choosing a parametrized map 
$u\colon \C\rightarrow M=M'\times \C$ over each element of $\M_{\gamma_i}(\alpha)$
(more precisely this gives a trivialization of the dual of $L$ and hence $L$).
Similar to the cylinder case, requiring that $\ev^{+}(u) = e^{i\theta_{\gamma_i}}$
kills the $S^1$ component of the automorphism group. To completely fix the complex structure
of the domain, require $u(1)\in M'\times \C$ to have modulus $1$ in the $\C$ component.

This trivialization induces a map $f\colon \M_{\gamma_i}(\alpha)\rightarrow \C$, 
given by the derivative at $0$ of $\pi\circ u$ at $0$. Then
$\int_{\mathcal{M}_{\gamma_i}}\ev^*(\theta)\wedge \psi$ is the number of zeroes of
$f$. Since $\M_{\gamma_i}(\alpha)$ is $2$-dimensional, its boundary consists of a 
collection of circles, and the number of zeroes of $f$ is equal to the winding number
of $f|_{\partial \M_{\gamma_i}(\alpha)}$ around $0$.

The boundary of $\M_{\gamma_i}(\alpha)$ consists of $2$-story curves
$(u_1, u_2)$ where $u_1$ is a holomorphic cylinder in the symplectization $\partial M\times \R$,
and $u_2$ is a holomorphic plane in $M$. The total index of the curve is $2$. There are two
possibilities depending on how the total index is distributed to the two levels.
\begin{enumerate}
\item
$(u_1,u_2)$ has index $(1,1)$: $u_1$ is a rigid (up to $\R$-translation) cylinder in 
$\M_{\gamma_i,\gamma'}$, and $u_2$ is in the $1$-parameter family of planes $\M_{\gamma'}(\alpha)$.
\item
$(u_1,u_2)$ has index $(2,0)$: $u_1$ is part of the $1$-parameter family of cylinders $\M_{\gamma_i,\gamma'}$,
and $u_2$ is a rigid planes in $\M_{\gamma'}(\alpha)$.
\end{enumerate}

A connected component of $\partial \M_{\gamma_i}(\alpha)$ can have curves of both types. Consider
a $3$-story curve $(v_1,v_2,v_3)$ where $v_1, v_2$ are rigid holomorphic cylinders in the symplectization
and $v_3$ is a rigid plane in the filling. If we glue $v_1$ and $v_2$ we have a curve of the first
type, and if we glue $v_2$ and $v_3$ we have a curve of the second type. This is the only way to pass from
one type to the other.

However in our special setting, there is in fact no curves of the first type in $\partial \M_{\gamma_i}(\alpha)$.
Recall the proof of Theorem \ref{main}, $\int_{\mathcal{M}_{\gamma_i}}\ev^*(\theta)\wedge \psi$ is
non-zero only if $\gamma_i = \gamma_p^2$ and $\indexof(p)=\dim(\alpha)$. By Proposition \ref{prop2},
$\gamma'=\gamma_q^2$ where $\indexof(q)=\indexof(p)+1$. But then $\M_{\gamma'}(\alpha)$ is empty
by part (4) of Proposition \ref{tech}.

Furthermore, if $(u_1,u_2)$ is a curve of the second type in $\partial \M_{\gamma_i}(\alpha)$, then
$u_1$ must be in $\M_{\gamma_i,\gamma'}^{S^1}$. If $u_1\in \M_{\gamma_i,\gamma'}^{[0,1]}$, then the 
boundary of the interval component that $u_1$ lies in, is a rigid broken cylinder. 
By the argument in the previous paragraph, this implies $\M_{\gamma'}(\alpha)$ is empty.

Therefore each component $S$ of $\M_{\gamma_i,\gamma'}^{S^1}$, together with a rigid plane 
$u_2\in\M_{\gamma'}(\alpha)$, gives a component $\widetilde{S}$ of $\partial \M_{\gamma_i}(\alpha)$. All
components of $\partial \M_{\gamma_i}(\alpha)$ arise this way.

The winding number of $f$ on such a boundary component $\widetilde{S}$ is the same as 
the winding number of $\ev^{-}$ on $S$. Since as $u_1$ varies on $S$, we need to parametrize 
the plane $u_2$ by rotation such that $\ev^{+}(u_2) = \ev^-(u_1)$, hence the derivative
at $0$ rotates accordingly. It follows that
$$\#\{\M_{\gamma_i}((\alpha,2))\} = \sum_{\overline{\mu}(\gamma')=\overline{\mu}(\gamma)-2} c_{\gamma_i,\gamma'} \#\{\M_{\gamma'}((\alpha,1))\}. $$
In other words
\[\int_{\mathcal{M}_{a}}\ev^*(\theta)\wedge \psi = \int_{\mathcal{M}_{D(a)}}\ev^*(\theta).
\qedhere\]
\end{proof}

Theorem \ref{coherent} immediately follows from Proposition \ref{D}.

\section*{Acknowledgements}
The author is deeply in debt to Y. Eliashberg for years of guidance, discussions
and ideas. He also thanks K. Honda, F. Bourgeois, and S. Lisi for many helpful 
suggestions and valuable comments.

\end{document}